\documentclass{amsart}
\usepackage{amsmath,amstext,amsthm,amsfonts,latexsym,amssymb,amscd} % Packages da AMS - American Math Society
\usepackage{graphicx,stmaryrd} 
\usepackage{eufrak}
\usepackage{abraces}
\usepackage{yhmath} 
\usepackage{mathtools,float}
\usepackage{qtree,forest,tikz-qtree}
\usepackage{color,fancybox}
\usepackage{empheq} 
\usetikzlibrary{fit} 
\usepackage{caption}
\usepackage{url}
\usepackage{bm} 

\DeclareCaptionType{InfoBox}
\usetikzlibrary{shapes}
\useforestlibrary{linguistics}
\usepackage{multirow,longtable}
\usepackage{multicol}

\usepackage{tikz}
\def\checkmark{\tikz\fill[scale=0.4](0,.35) -- (.25,0) -- (1,.7) -- (.25,.15) -- cycle;}

\theoremstyle{plain}
\newtheorem{theorem}{Theorem}[section]
\newtheorem{proposition}[theorem]{Proposition}
\newtheorem{lemma}[theorem]{Lemma}
\newtheorem{corollary}[theorem]{Corollary}
\newtheorem{remark}{Remark}
\newtheorem{example}[theorem]{Example}

\newcommand{\field}[1]{\mathbb{#1}}
\newcommand{\RR}{\field{R}}
\newcommand{\ZZ}{\field{Z}}
\newcommand{\N}{\field{N}}

\newcommand{\Fib}{\operatorname{Fib}}
\newcommand{\dd}{\operatorname{d}}
\DeclareMathAlphabet{\mymathbb}{U}{BOONDOX-ds}{m}{n}  % double one

\providecommand{\keywords}[1]
{
  \small	
  \textbf{\textit{Keywords---}} #1
}

\newcommand\blfootnote[1]{%
  \begingroup
  \renewcommand\thefootnote{}\footnote{#1}%
  \addtocounter{footnote}{-1}%
  \endgroup
}

\begin{document}

\title{Entropy for $k$-trees defined by $k$ transition matrices}

\author{Alexandre Baraviera}
\address{Instituto de Matem\'{a}tica e Estat\'{i}stica, UFRGS, Av. Bento Gon\c{c}alves 9500, 91500, Porto Alegre, RS, Brazil.}
\email{atbaraviera@gmail.com}
%\thanks{}

\author{Alex Becker}
\address{Centro de Ci\^{e}ncias Naturais e Exatas, UFSM, Av. Roraima 1000, Santa Maria, RS, Brazil.}
\email{alex.becker@ufsm.br}
%\thanks{}

\author{Andressa Cordeiro}
\address{Instituto de Matem\'{a}tica e Estat\'{i}stica, UFRGS, Av. Bento Gon\c{c}alves 9500, 91500, Porto Alegre, RS, Brazil.}
\curraddr{ISEA, Universit\'{e} de la Nouvelle-Cal\'{e}donie, Noum\'{e}a, New Caledonia.}
\email{andressap.ha@gmail.com}

\date{\today}

 \begin{abstract}
     We study Markov tree-shifts given by $k$ transition matrices, one for each of its $k$ directions. We provide a method to characterize the complexity function for these tree-shifts, used to calculate the tree entropies defined by Ban and Chang \cite{BanChangEntropy}, and Petersen and Salama \cite{PetersenSalama1}. Moreover, we compare these definitions of entropy in order to determine some of their properties. The characterization of the complexity function provided is used to calculate the entropy of some examples. The question of existence of a specific type of invariant measures for such tree-shifts is addressed. Finally, we analyse some topological properties introduced by Ban and Chang \cite{BanChangChaos} for the purpose of answering two of the questions raised by these authors.
 \end{abstract}


\maketitle

\blfootnote{This study was financed in part by the Coordenação de Aperfeiçoamento de Pessoal de Nível Superior – Brasil (CAPES) – Finance Code 001, and within the scope of the Program CAPES/COFECUB. It was conducted during the third author's doctoral program.}

\section{Introduction}

A tree-shift displays properties of both one-dimensional and higher dimensional shift spaces, behaving as an intermediate class of symbolic dynamics. In addition, from a certain point of view, shifts on trees are a natural generalization of the full shift $\Sigma_d$ on the one-dimensional space of sequences $x=(x_n)_{n\geq0}$ over the alphabet 
$\mathcal{A} = \{ 0, \dots, d-1 \}$: given $x \in \Sigma_d$, there is only one possible ``forward'' position at any $x_n$, namely, $x_{n+1}$. 
Now, fix an integer $k > 1$ and allow each position to be followed along one of $k$ directions. This is basically what a $k$-tree represents. This structure is, for example, a branch of research in the context of machine learning and $k$-ary decision languages. See \cite{ArtigoProgramacao} and references therein.

Since the works of Aubrun and Béal \cite{AubrunBeal,AubrunBealSofic}, shifts on trees has received substantial attention: Ban and Chang \cite{BanChangEntropy}, and Petersen and Salama \cite{PetersenSalama1}, defined entropies for tree-shifts, as we now denote as $ h_{BC} $ and $ h_{PS} $, respectively, in terms of a complexity function, in a natural generalization of the entropy of $ \Sigma_d $. Furthermore, in \cite{BanChangChaos}, the authors generalized for tree-shifts the concepts of irreducibility, mixing and chaos in the sense of Devaney and provided many results relating these properties.

When it comes to entropy, works like \cite{BanChang2020,BanChangEtAll2021,BanChangEtAll2022,PetersenSalama1,PetersenSalama2} considered $k$-trees whose paths are generated by a shift space $\Sigma_d$, called \textit{hom tree-shifts}. In the case that $\Sigma_d = \Sigma_M$ is a Markov shift given by a $0,1$ transition matrix $M$, the allowed transitions for trees of the corresponding tree-shift $\mathcal{T}_M$ are also given by $M$, in all directions. As a consequence, in the tree-shifts aforementioned there are ``symmetrical elements'', or, in other words, there exist trees $t$ that have $k$ copies of a tree $t^\prime$ attached to its root.

In the case that $M$ is irreducible, the entropy $h_{PS}$ of both $\Sigma_M$ and $\mathcal{T}_M$ are equal if, and only if, the sum of all rows of $M$ are the same, as proved in \cite{BanChang2020}. An irreducible $M$ is also important to prove that the entropy of $ \mathcal{T}_M $ can be more than the maximum between the entropy of its corresponding irreducible blocks, while in dimension one, the equality holds. See \cite{BanChangEtAll2021}.

In this work, we consider tree-shifts whose elements are $k$-trees given by $k$ (possibly different) $d \times d $ matrices. The entries of these matrices can be $0$ (the correspondent transition is not allowed) or $1$ (the correspondent transition is allowed). In this context, $d$ is the cardinality of the alphabet $\mathcal{A}$. One can see that there are $k$-trees given by different matrices that cannot be associated to tree-shifts given by a single matrix for all directions. For this purpose, the reader is invited to follow the analysis, for example, of the binary tree-shift $X_{20} = (D,E)$ defined in Section \ref{sectionentropytable}. 

The main purpose of this text is to construct a method to calculate the characteristic function $p(n)$, inspired in \cite{PetersenSalama1}. Moreover, we prove properties of $h_{BC}$ and $h_{PS}$ for Markov tree-shifts given by $k$ transition matrices and calculate those entropies in some examples. The examples are important not only to understand our algorithm to determine entropies, but also to complement 
the work \cite{BanChangEntropy}, where Ban and Chang proposed four open problems, two of them relating entropy and topological properties of tree-shifts.

This study is organized in eight principal sections. The introduction is presented in Section 1 and the main definitions, in Section 2. Afterwards, Section 3 is dedicated to introduce an algorithm to calculate the complexity function of a Markov tree-shift given by transition matrices, what allows us to determine its entropy by an explicit (although of recurrence) relation. The following section compares the entropies proposed by Ban and Chang in \cite{BanChangEntropy}, and Petersen and Salama in \cite{PetersenSalama1}, in terms of topological invariance and maximum value reached. In Section 5, we give upper and lower bounds for the entropy $h_{PS}$ in terms of the matrices that determine the allowed transitions of a tree-shift. Then, in Section \ref{sectionentropytable}, we calculate the entropy of all the binary tree-shifts (given by two transition matrices) over the alphabet $\mathcal{A} = \{0,1\}$ in order to exhibit how our algorithm can be applied. Section \ref{SectionInvariantMeasures} presents the definition of a measure for tree-shifts, inspired by the context of Markov chains, and we prove a condition to the existence of a such invariant measure for the tree-shifts addressed in this work. Finally, in Section \ref{sectiontopologicalproperties} we present the definitions of irreducibility, mixing and chaos in the sense of Devaney for tree-shifts given by transition matrices, following the work of Ban and Chang \cite{BanChangChaos}, with a view to provide some discussions around the subject and answer questions proposed in their work.

\section{Notation and main definitions}

Let $\Sigma = \{a_1, \dots, a_k\}$ be the set of the $k$ generators of the free monoid $\Sigma^*$ with the operation of concatenation. The elements of $\Sigma^*$ are the finite words of any length and the empty word $\epsilon$. Given a finite alphabet $\mathcal{A} = \{0, \dots, d - 1\}$, a \textit{labeled $k$-tree} $t$, or simply $k$-tree, is a map from $\Sigma^*$ to $\mathcal{A}$. Any word of $\Sigma^*$ is called a \textit{node} of a tree $t$, and $\epsilon$ corresponds to its root. Each node $x \in \Sigma^*$ has children $xi$, with $i \in \Sigma$, and the label at node $x$ of $t$ is denoted by $t_x$. A labeled node is at \textit{level} $n$ of a tree if the length of the correspondent node is $n$. We use the notation $wx$ to represent the concatenation of the words $w$ and $x$, and $z^j = \underbrace{z \dots z}_{j \text{ times}}$, i.e., $z^j$ denotes the concatenation of $z \in \Sigma^*$ $j$ times.

We define $\mathcal{T} (\mathcal{A})$ as the set $\mathcal{A}^{\Sigma^*}$ of all $k$-trees over $\mathcal{A}$. In particular, if $k = 2$ we say that $t \in \mathcal{T}(\mathcal{A})$ a \textit{binary tree}. For $t, t^\prime \in \mathcal{T}(\mathcal{A})$ we define the metric $\dd (t, t^\prime) = \frac{1}{2^\ell}$, where $\ell$ is the shorter length of a word $x$ in $\Sigma^*$ such that $t_x \neq t_{x}^\prime$. It is clear that $\dd(t, t) = 0$. This metric induces in $\mathcal{T} (\mathcal{A})$ a topology that is equivalent to the usual product topology, which guarantees that $\mathcal{T}(\mathcal{A}) $ is a compact set. For each $i \in \Sigma$, define the $i$\textit{-th shift map} $\sigma_i$ from $\mathcal{T}(\mathcal{A})$ to itself such that the image of $t$ is the tree rooted at its $i$-th child, that is, $\sigma_i(t)_x = t_{ix}$ for all $x\in\Sigma^*$. The compact metric  space $\mathcal{T}(\mathcal{A})$ equipped with the maps $\sigma_i$ is called the \textit{full tree-shift} over $\mathcal{A}$.

A \textit{block of length} $n$, denoted by $\Delta_n = \cup_{i=0}^n \Sigma^i$, is the initial subtree of height $n$ of a tree in $\mathcal{T}(\mathcal{A})$, $n \geq 0$. In other words, it is the set of all labeled nodes at level $j$ with $0 \leq j \leq n$. It has a total of $(k^{n+1}-1)/(k-1)$ nodes. A \textit{tree-subshift} $ X $ of $ \mathcal{T}(\mathcal{A}) $ is the set $ X_{\mathcal{F}} $ of all trees that do not contain any block of a set of \textit{forbidden blocks} $\mathcal{F}$; in this case, $ \sigma_i(X) \subset X $ for each $i \in \Sigma$. A block is \textit{allowed} (or admissible) if it is not forbidden. If $ \mathcal{F} $ is a finite set, we say that $ X_{\mathcal{F}} $ is a \textit{tree-subshift of finite type} and we can assume that all of its blocks have the same length. In the case that this length is one, $X_{\mathcal{F}}$ is a \textit{Markov tree-subshift}. This study considers a specific class of Markov tree-subshifts: all the allowed transitions can be established by $k$ matrices of order $d$ whose entries are $0$ or $1$. 

For any nonnegative integer $m$, denote by $\mathcal{L}_m(X)$ the set of all allowed blocks in $X$ of length $m$. Define a function $ \phi : \mathcal{L}_m(X) \rightarrow \mathcal{A}^\prime $, where $ \mathcal{A}^\prime $ is a finite alphabet. For any $t \in X $, denote by $ b(x) \in \mathcal{L}_m(X) $ the block of length $ m $ of $t$ rooted at $ x \in \Sigma^* $. We say that a function $ \Phi : X \rightarrow \mathcal{T} (\mathcal{A}^\prime) $ is a $ m $\textit{-block map} given by $\phi$ if $ \Phi(t)_x = \phi(b(x)) $ for all $ x \in \Sigma^* $ and $ t \in X $.

Now, suppose that $ X = X_{\mathcal{F}} $ is a tree-shift of finite type and let $\tilde{m}$ be the length of the blocks of $\mathcal{F}$. In particular, $ \mathcal{L}_{\tilde{m}}(X) $ is a set of blocks that determine all trees $ t \in X $. Therefore, we can construct $ \phi : \mathcal{L}_{\tilde{m}}(X) \rightarrow \mathcal{A}^\prime $ bijective for a suitable $\mathcal{A}^\prime$ and a $\tilde{m}$-block map $ \Phi : X \rightarrow Y $, where $ Y = \Phi(X) \subset \mathcal{T} (\mathcal{A}^\prime) $. In this case, we say that $X$ and $Y$ are \textit{conjugated}. Moreover, $Y$ is a Markov tree-shift given by $k$ transition matrices, and the allowed transitions in each direction of trees in $Y$ are constructed in terms of the possible overlaps between admissible blocks of $X$. In this text, a transition matrix $M$ is a square matrix with all entries $M(i,j) \in \{0, 1\}$. Thus it is proved the following result:

\begin{proposition}\label{PropConjugationTreeShiftsFiniteTypeandMarkov}
Any tree-shift of finite type consisting of $k$-trees, up to an alphabet change by a suitable $m$-block map, is a Markov tree-shift whose allowed blocks can be characterized by matrices $A_1, \dots, A_k$, all with the same order.
\end{proposition}

In
%Proposition 4.2 chineses
\cite{AubrunBeal,BanChangChaos} the authors considered $\phi$ onto but not necessarily injective. To relate the entropy of a tree-shift of finite type and the entropy of the conjugated Markov tree-shift given by transition matrices, it is apt to consider a bijection between $\mathcal{L}_{\tilde{m}}(X)$ and a new alphabet $\mathcal{A}^\prime$. In \cite{BanChangChaos}, it is introduced an equivalent result to Proposition \ref{PropConjugationTreeShiftsFiniteTypeandMarkov} in terms of graphs.

When $X$ is a tree-shift given by matrices as in Proposition \ref{PropConjugationTreeShiftsFiniteTypeandMarkov}, we denote $ X = (A_1, \dots, A_k) $. We assume that each row and column of $ A_1, \dots, A_k $ has at least one $1$. With this, all elements in $\mathcal{A}$ appear on trees of $X$, and from any element of the alphabet, there exists at least one possible transition to another element. This is the meaning of the previously employed expression ``all with the same order'' in the Preposition above, this order being the cardinality of $\mathcal{A}$. 

Any $ d \times d $ transition matrix $A$ is said to be \textit{irreducible} if, for every pair $ i,j \in \{1, \dots, d\}$, there exists $n \geq 1$ such that $A^n(i,j) > 0$, or, in other words, the entry of $A^n$ in row $i$ and column $j$ is strictly positive. If there exists an $n$ satisfying this property for all $i, j$, then $A$ is called \textit{aperiodic}.

\section{Characterization of the complexity function}\label{sectionmethod}

In this section we analyze the complexity function of tree-shifts consisting of $k$-trees, which requests a way of counting the number of allowed configurations with $n$ levels, $n \in \N$. Following \cite{PetersenSalama1} we use a recursive method to obtain this number, that is itself a dynamical system.

Let $ X $ be a tree-shift over the alphabet $ \mathcal{A} = \{0, \dots, d - 1\} $ and define $ p $ the \textit{complexity function} of $X$ as $ p(n) = \# \mathcal{L}_n(X) $, $ n \geq 0 $. In essence, $p(n)$ is the number of allowed blocks of length $n$ of trees in $X$. This function is crucial to determine the entropy of $ X $, defined by Ban and Chang \cite{BanChangEntropy}, and Petersen and Salama \cite{PetersenSalama2}, respectively, by
\begin{equation}\label{EquationEntropies}
   h_{BC}(X) = \lim\limits_{n\rightarrow\infty} \dfrac{\log\log p(n)}{n}
\quad \quad \quad \text{and} \quad \quad \quad
h_{PS}(X) = \lim\limits_{n\rightarrow\infty} \dfrac{\log p(n)}{1 + k + \dots + k^n}.
\end{equation}
We present a function $ f $ and relate its $n$-th iterated to $p(n)$, to provide a method to calculate $ h_{BC} $ and $ h_{PS} $ in the case where $ X $ is a Markov tree-shift given by $k$ transition matrices.

Henceforth, we consider $ X = (A_1, \dots, A_k)$. Define a dynamical system $ f: \RR_+^d \rightarrow \RR_+^d $ by the equation
$$
f(x) = (A_1 x)*(A_2 x) * \dots * (A_k x),
$$
where $*$ denotes the product $ (x_1, x_2, \dots, x_d) * (y_1, y_2, \dots, y_d) = (x_1 y_1, x_2 y_2, \dots, x_d y_d) $. Hence,
\begin{equation}\label{eqi-thCoordinateFunctionf}
f(x)_i =
(A_1x)_i \dots (A_kx)_i =
\left(\sum_{j=1}^d (A_1)_{ij} x_j\right) \dots \left(\sum_{j=1}^d (A_k)_{ij} x_j\right),
\end{equation}
where $ f(x)_i $ denotes the $i$-th coordinate of $ f(x)$. The function $f$ an homogeneous of degree $k$, since, for any positive real number $\lambda$, we have $ f(\lambda x) = \lambda^k f(x) $.

Consider the equivalence classes defined by the straight lines passing through the origin: we identify the vectors $x$ and $y$ if there exists some positive $\lambda$ such that $x = \lambda y$. Since $f$ is homogeneous it preserves the equivalence classes and we can observe the behavior of the dynamics of $f$ on the quotient space. In this space, for $x_d \neq 0$, we use the change of coordinates $ \eta_i = \frac{x_i}{x_d} $, $ i = 1, \dots, d $, and for $ \eta = (\eta_1,\dots,\eta_{d-1},1) $ we define $ N : \RR^d_+ \rightarrow \RR_+ $ and $ F : \RR^d_+ \rightarrow \RR^d_+ $ by
$$
N(\eta) =
\left(\displaystyle\sum\limits_{j=1}^d (A_1)_{dj}\eta_j\right) \left(\displaystyle\sum\limits_{j=1}^d (A_2)_{dj}\eta_j\right)  \dots  \left(\displaystyle\sum\limits_{j=1}^d (A_k)_{dj}\eta_j\right)
$$
and
$$
F(\eta)_i = \dfrac{f(x)_i}{f(x)_d} =
\dfrac{\left( \displaystyle\sum\limits_{j=1}^d (A_1)_{ij}\eta_j \right)\dots\left( \displaystyle\sum\limits_{j=1}^d (A_k)_{ij}\eta_j \right)}{N(\eta)} ,
$$
where $ i = 1,\dots,d $. Notice that $ F(\eta) = (F(\eta)_1, \dots, F(\eta)_{d-1},1) $ by definition. We then obtain, for $ x =(x_1, \dots, x_d) \in \RR^d_+ $,
$$
f(x) = f(x_d \eta) = x_d^k N(\eta) F(\eta),
$$
and, inductively, for all $ n \in \N $,
$$
f^n(x) = f^n (x_d \eta) = x_d^{k^n} N(\eta)^{k^{n-1}} N(F(\eta))^{k^{n-2}} N(F^2(\eta))^{k^{n-3}} \dots N(F^{n-1}(\eta)) F^n(\eta).
$$

Considering the initial condition as $ x_1 = \dots = x_d = 1 $, we get $ \eta_i = 1 $ for $ i = 1, \dots, d $ and
\begin{equation}\label{eqdefinitionfunctionf}
 f^n(\mymathbb{1}) = N(\mymathbb{1})^{k^{n-1}} N( F(\mymathbb{1}))^{k^{n-2}} N( F^2(\mymathbb{1}))^{k^{n-3}} \dots N(F^{n-1}(\mymathbb{1})) F^n(\mymathbb{1}),
\end{equation}
where $ \mymathbb{1} = (1,\dots,1) \in \RR^d_+ $.

The following result uses arguments inspired in Theorem 1.6 of \cite{BanChangEntropy} to prove how the iterates of $f$ are useful to determine the complexity function of a tree-shift $X = (A_1, \dots, A_n)$.

\begin{theorem}\label{Prop-p(n)=f^n}
For all $ n \geq 1 $ we have $ p(n) = \|f^n(\mymathbb{1})\| := \sum_{j=1}^d f^n(\mymathbb{1})_j $, where $f^n(\mymathbb{1})$ is given in equation \eqref{eqdefinitionfunctionf}.
\end{theorem}

\begin{proof}
First, we calculate the number of blocks of length $1$ with root $i \in \mathcal{A}$. Identify the leftmost child of the root of a block as the first child, and proceed in this way until the rightmost node, considered the $k$-th child. The number of possible labels to the $\ell$-th child is given by the sum of the $i$-th row of $A_\ell$. Multiplying the possibilities for all $ k $ children, we obtain the number of allowed blocks of length $1$ with root $i$ as
$$
(A_1\mymathbb{1})_i \dots (A_k\mymathbb{1})_i = \left(\sum_{j=1}^d (A_1)_{ij}\right) \dots \left(\sum_{j=1}^d (A_k)_{ij}\right) = f(\mymathbb{1})_i,
$$
and it implies that $ \|f(\mymathbb{1})\| = p(1) $.

Analogously, one can calculate the number of allowed blocks of length $ n+1 $ with root $i$ by relating the number of allowed blocks of length $1$ starting on $ i $ and the number of allowed blocks of length $n$ as follows. Let $ \delta_{ij}^m = 1 $ if $ i $ can be followed by $j$ on the $ m $-th children and $ \delta_{ij}^m = 0 $ otherwise, $1 \leq m \leq k$, and denote by $ B_n(j) $ the number of allowed blocks of length $n$ starting on $j$. The number of allowed blocks of length $ n+1 $ starting on $i$ is
\begin{equation}\label{eqInductionHypothesisComplexityFunction}
   B_{n+1}(i) = \left( \displaystyle\sum\limits_{j=1}^d \delta_{ij}^1 B_n(j) \right) \dots \left( \displaystyle\sum\limits_{j=1}^d \delta_{ij}^k B_n(j) \right). 
\end{equation}

Now, suppose that, for some $ n\geq1 $, $ B_n(j) = f^n(\mymathbb{1})_j $ for all $ 1\leq j\leq d $. Since $ \delta_{ij}^m = (A_m)_{ij} $, by equations \eqref{eqi-thCoordinateFunctionf} and \eqref{eqInductionHypothesisComplexityFunction} we get
$$
B_{n+1}(i) = \left( \displaystyle\sum\limits_{j=1}^d (A_1)_{ij} f^n(\mymathbb{1})_j \right) \dots \left( \displaystyle\sum\limits_{j=1}^d (A_k)_{ij} f^n(\mymathbb{1})_j \right)
= f(f^n(\mymathbb{1}))_i = f^{n+1}(\mymathbb{1})_i.
$$
Thereby, since $B_1(j) = f(\mymathbb{1})_j$, $1 \leq j \leq d$, the proof of $ \|f^n(\mymathbb{1})\| = p(n) $ for all $ n\geq1 $ follows by induction.
\end{proof}

By Theorem \ref{Prop-p(n)=f^n}, we can write the definitions of entropy in equation \eqref{EquationEntropies} as
$$
h_{BC}(X) = \lim\limits_{n\rightarrow\infty} \dfrac{\log \log \|f^n(\mymathbb{1})\|}{n}\ \ \ \text{and}\ \ \ h_{PS}(X) = \lim\limits_{n\rightarrow\infty} \dfrac{\log\|f^n(\mymathbb{1})\|}{1 + k + \dots + k^n}.
$$

\begin{remark}\label{RemarkCoordinatesChaange}
One could choose any $ 1 \leq \ell < d $ to define the change of coordinates $ \tilde{\eta}_i = x_i / x_\ell $, $i = 1, \dots, d$, $x_\ell \neq 0$, and the corresponding functions $ \tilde{N} $ and $ \tilde{F} $ to obtain a different but equivalent expression to $ f^n(\mymathbb{1}) $.
\end{remark}

\begin{remark}\label{RemarkEquivalentTreeShiftsSameEntropy}
Two Markov tree-shifts consisting of $k$-trees that are the same up to an one-to-one correspondence of their alphabets have the same complexity function, and, consequently, the same entropies $ h_{BC} $ and $ h_{PS} $. Tree-shifts determined by the same $k$ matrices but in different order also have the same entropy.
\end{remark}

\section{Comparing the entropies $ h_{BC} $ and $ h_{PS} $ for tree-shifts}

In this section, we aim to show that $ h_{PS} $ is not preserved by conjugation and we give a characterization to the relation between the entropy $h_{PS}$ of a tree-shift of finite type and its conjugated Markov tree-shift. Moreover, we give upper bounds to both definitions of entropy for any tree-shift, in relation to $k$ and $d$.

\begin{theorem}\label{TheoNotTopologicalInvariant}
The entropy $ h_{PS} $ is not a topological invariant.
\end{theorem}

\begin{proof}
Consider $ X = X_{\mathcal{F}} $ a tree-shift of finite type consisting of $k$-trees and $ s $ the length of all forbidden blocks of $ \mathcal{F} $. The allowed blocks of $ X $ can be described in terms of a finite set $ D $ consisting of blocks of length $s$. Let $ \mathcal{A}^\prime = \{0,\dots,|D|-1\} $ be an alphabet and $ Y \subset \mathcal{T}(\mathcal{A}^\prime) $ the Markov tree-shift such that $ \Phi : X \rightarrow Y $ is a conjugation defined by a bijection $ \phi : D \rightarrow \mathcal{A}^\prime $. Denote by $ p_X $ and $ p_Y $ the complexity functions of $X$ and $Y$, respectively. For all $ n \geq 1 $,
$$
p_X (n+s) = p_Y (n).
$$

Using the relation $1 + k + \dots + k^n = (k^{n+1}-1)/(k-1)$ and the fact that $Y$ also consists of $k$-trees, we get
$$
h_{PS}(Y)
= \lim\limits_{n\rightarrow\infty} \dfrac{k-1}{k^{n+1}} \log p_Y(n)
= k^s \lim\limits_{n\rightarrow\infty} \dfrac{k-1}{k^{n + s + 1} } \log p_X(n+s)
= k^s \ h_{PS}(X).
$$
Whenever $ h_{PS}(X) \neq 0 $ we have $ h_{PS}(X) \neq h_{PS}(Y) $. 
\end{proof}

\begin{remark}\label{Remarkh_BCisInvariantFiniteType}
Let $X$ be a tree-shift of finite type consisting of $k$-trees with $h_{BC}(X) = \theta$ for some $\theta \geq 0$ and suppose that $X$ is conjugated to a tree-shift $Y$ by a function $\Phi$ as defined previously. Then $h_{BC}(X) = h_{BC}(Y)$.
\end{remark}

\begin{proof}
Suppose that, as well as in Theorem \ref{TheoNotTopologicalInvariant}, we have $p_X(n+s) = p_Y(n)$ for some $s \geq 1$. Given $\varepsilon > 0$, there exists $n_0 \in \N$ such that, for $n > n_0$,
$$
e^{n(\theta - \varepsilon)} < \log p_X (n) < e^{n(\theta + \varepsilon)}.
$$
Then, $h_{BC}(Y) = \lim\limits_{n\rightarrow \infty} \frac{1}{n} \log \log p_X(n+s) \leq \theta + \varepsilon$ and $h_{BC}(Y) \geq \theta - \varepsilon$. Since $\varepsilon > 0$ is arbitrarily small, we have $h_{BC}(X) = h_{BC}(Y)$.
\end{proof}

\begin{remark}
%\textcolor{blue}{
Aubrun and Béal presented in \cite{AubrunBeal} an example of conjugacy between two Markov tree-shifts showing that $ \phi $ does not necessarily need to be bijective. In the general case (for conjugation maps such that $\phi$ is not bijective), it is not clear how to compare the complexity functions of two conjugated tree-shifts. Therefore, it is not determined whether $h_{BC}$ is a topological invariant or not.%}
\end{remark}

We now establish a difference between $h_{BC}$ and $h_{PS}$ in regard to their upper bonds.

\begin{proposition}\label{PositivePSentropyImpliesBCTotalEntropy}
Let $X$ be a tree-shift consisting of $k$-trees over $ \mathcal{A} = \{0,\dots,d-1\} $. Then
$$
0 \leq h_{PS}(X) \leq \log d
\ \ \ \text{and}\ \ \
0 \leq h_{BC}(X) \leq \log k.
$$
Also, whenever $h_{PS}(X) > 0$, we get $h_{BC}(X) = \log k $.
\end{proposition}

\begin{proof}
Let $Y$ be the tree-shift over $\mathcal{A}$ given by the matrices
$$
A_1 = A_2 = \dots = A_k  =
\left[
  \begin{array}{ccc}
    1 & \dots  &   1   \\
    \vdots & \ddots & \vdots \\
    1 & \dots &   1
  \end{array}
  \right].
$$
Then, $ F^n(\eta)_i = 1 $ for all $ 1 \leq i \leq d $, $ n \in \N $ and $ N(\eta) = (\eta_1 + \dots + \eta_{d-1} + 1)^k $, what implies that $ N(F^n(\mymathbb{1})) = d^k $ for all $ n $. It yields
$$
\|f^n(\mymathbb{1})\| = d^{k^n} d^{k^{n-1}} \dots d^k d.
$$
As a consequence,
$$
h_{PS}(X) = \lim\limits_{n\rightarrow\infty} \left( \dfrac{1+k+\dots+k^n}{1+k+\dots+k^n} \log d \right) = \log d
$$
and
$$
h_{BC}(X) = \lim\limits_{n\rightarrow\infty} \dfrac{1}{n} \log\log (d^{k^n}d^{k^{n-1}}\dots d^k d) = \lim\limits_{n\rightarrow\infty} \dfrac{1}{n} \log \left( \dfrac{k^{n+1}-1}{k-1} \log d \right) = \log k.
$$
The equation $h_{BC}(X) = \log k$ is exposed in \cite{BanChangEntropy} in other terms. Notice that $Y$ is the full tree-shift, therefore, it has the maximum entropy in both cases since the matrices considered generate the maximum number of allowed blocks of length $1$. Id est, any tree-shift $X$ over the same alphabet is a restriction of $Y$, and consequently, has at most the same number of allowed blocks of length $n$ as $Y$, for each $n \in \N$.

Now, suppose that $X$ is a tree-shift such that
$$
h_{PS}(X) = \lim\limits_{n\rightarrow\infty} \dfrac{k-1}{k^{n+1}} \log \|f^n(\mymathbb{1})\| = c > 0.
$$
By the definition of limit, for any $\varepsilon > 0$ there exists $n_0 \in \N$ such that, for $n > n_0$,
$$
\dfrac{k^{n+1}}{k-1} (c - \varepsilon) < \log \| f^n(\mymathbb{1}) \|.
$$
%for all $n > n_0$. 
For that reason, we get
$$
h_{BC}(X) \geq \lim\limits_{n \rightarrow \infty} \dfrac{1}{n} \log \left( \dfrac{k^{n+1}}{k-1}(c - \varepsilon) \right) = \log k.
$$
Since $\log k$ is an upper bound for $h_{BC}(X)$, the proof is thus shown.
\end{proof}

\begin{remark}
    The previous result does not require $X$ to be a tree-shift of finite type.
\end{remark}

By Proposition \ref{PositivePSentropyImpliesBCTotalEntropy} there is no interest in calculating $h_{BC}$ for tree-shifts with positive $h_{PS}$. It is thus natural to question the behavior of $h_{BC}$ for tree-shifts with $h_{PS} = 0$. In \cite{BanChangEntropy} the authors establish a method to generate tree-shifts of finite type with entropy $h_{BC}$ in the interval $(0, \log k)$ in terms of recurrence sequences. Since they did not construct explicitly the allowed blocks of a tree-shift satisfying their conditions, we do so, in the aim of showing that we should not ignore the class of tree-shifts with zero $h_{PS}$.

\begin{example}\label{ExampleZerohPS-PositivehBC}
Inspired by Example 3.2 of \cite{BanChangEntropy}, we consider $X$ the tree-shift determined by the following admissible blocks:

\begin{center}
\begin{forest}
[{$0$}
[{$0$}] [{$1$}]]
\end{forest}
 \ \ \ \ \ \ \ \ 
\begin{forest}
[{$0$}
[{$1$}] [{$0$}]]
\end{forest}
 \ \ \ \ \ \ \ \ 
\begin{forest}
[{$1$}
[{$0$}] [{$2$}]]
\end{forest}
 \ \ \ \ \ \ \ \ 
\begin{forest}
[{$2$}
[{$2$}] [{$2$}]]
\end{forest}
\end{center}

This tree-shift satisfies the recurrence equations of the example in question, with $a^{(0)} = 0$, $a^{(1)} = 1$ and $b = 2$, therefore, we should have $h_{BC}(X) = \log \frac{1 + \sqrt{5}}{2}$. We indicate explicitly how to use the method presented in Section \ref{sectionmethod} to find this value. Associate the first block with $0^\prime$, the second block with $1^\prime$, and the third and fourth blocks with $2^\prime$ and $3^\prime$, respectively. Consider the new alphabet $\mathcal{A}^\prime = \{ 0^\prime, 1^\prime, 2^\prime, 3^\prime \}$. Notice that $X$ is conjugated to the Markov tree-shift $Y = (A_1, A_2)$ over $\mathcal{A}^\prime$, with
$$
A_1 =
\begin{pmatrix}
    1 & 1 & 0 & 0 \\
    0 & 0 & 1 & 0 \\
    1 & 1 & 0 & 0 \\
    0 & 0 & 0 & 1
\end{pmatrix}
\ \ \text{ and }\ \ 
A_2 = 
\begin{pmatrix}
    0 & 0 & 1 & 0 \\
    1 & 1 & 0 & 0 \\
    0 & 0 & 0 & 1 \\
    0 & 0 & 0 & 1
\end{pmatrix}.
$$
The eigenvalues of $A_1$ are $0, 1, \frac{1 + \sqrt{5}}{2}$ and $\frac{1 - \sqrt{5}}{2}$ and the eigenvalues of $A_2$ are $0$ and $1$.

For $Y$, we obtain $f(x_1, x_2, x_3, x_4) = ((x_1 + x_2) x_3, (x_1 + x_2) x_3, (x_1 + x_2) x_4, x_4^2)$ and, considering $\eta_i = \frac{x_i}{x_4}$, we get $F(\eta_1, \eta_2, \eta_3, 1) = ((\eta_1 + \eta_2) \eta_3, (\eta_1 + \eta_2) \eta_3, \eta_1 + \eta_2, 1)$ and $N(\eta) = 1$. Denote $ (\Fib(n))_{n \geq 1} = (1,1,2,3,5,8,\dots) $ the Fibonacci sequence. Then, $F^n(1)_1 = 2^{\Fib(n+3) - 2}$ and $F^n(1)_3 = 2^{\Fib(n+2)-1}$. It yields 
$$
\| f^n(\mymathbb{1}) \| = 2^{\Fib(n+3) - 1} + 2^{\Fib(n+2) - 1} + 1.
$$
Using the fact that $2^{\Fib(n+3) - 1} \leq \| f^n(\mymathbb{1}) \| \leq 2^{\Fib(n+3)}$, we prove that $h_{PS}(Y) = 0$ and $h_{BC}(Y) = \log \frac{1 + \sqrt{5}}{2}$. By Remark \ref{Remarkh_BCisInvariantFiniteType}, $h_{BC}(X) = h_{BC}(Y)$ as desired. 
\end{example}

\section{Upper and lower bounds for the entropy $h_{PS}$ of Markov tree-shifts}

In the previous section we presented an upper bound for the entropy $ h_{PS} $ of any tree-shift. In this section, we consider a tree-shift of the form $ X = (A_1, \dots, A_k) $ to provide a second upper bound and lower bounds depending on some properties of the matrices $ A_1, \dots, A_k $.

First, we fix the vector norm
$$
\| v \|_m  = \|(v_1, \dots, v_d) \|  =  \max_{i \in \{1, \dots, d\}}{|v_i|}
$$
and the corresponding operator norm $\| \cdot  \|_{op} $ for matrices, defined as
$$
  \| M \|_{op} = \sup_{\|v \|_m=1}{\|M v \|_m}.
$$
If $M$ has non-negative entries, this norm corresponds to the
maximum among the sum of the elements in each row of the matrix $ M $.

\begin{remark}\label{RemarkAperiodicTransitionMatrix}
    If $M$ is a $d \times d$ aperiodic transition matrix, then $\| M \|_{op} = \gamma \geq 2$. 
\end{remark}

Considering this norm, we have the following:

\begin{proposition}\label{PropUpperBondEntropyPS}
Let $ X = (A_1,\dots,A_k) $ be a tree-shift over the alphabet $ \mathcal{A} = \{ 0,\dots,d-1 \} $. Then,
$$
   h_{PS}(X) \leq \dfrac{1}{k} ( \log d + \log \|A_1\|_{op} + \dots + \log \|A_k\|_{op} ).
$$
\end{proposition}

\begin{proof}
For a point $ x $ with non-negative entries we have $|(A_ix)_j| \leq \|A_i\|_{op}\ \|x\|_m$. Using the definition of $f$, we obtain $\|f(\mymathbb{1})\| \leq d \ \|A_1\|_{op}\|A_2\|_{op}\dots\|A_k\|_{op}$. Now, by iterating $f$, for each $ n \in \N $ we get
$$
 \| f^n(\mymathbb{1}) \| \leq d^{1+k+\dots+k^{n-1}}\|A_1\|_{op}^{1+k+\dots+k^{n-1}}\dots\|A_k\|_{op}^{1+k+\dots+k^{n-1}}.
$$
Therefore,
$$
h_{PS}(X) \leq \dfrac{1}{k} \big(\log d + \log \|A_1\|_{op} + \dots + \log \|A_k\|_{op} \big).
$$
\end{proof}

Using Propositions \ref{PositivePSentropyImpliesBCTotalEntropy} and \ref{PropUpperBondEntropyPS} for $X = (A_1,\dots,A_k)$ over $\mathcal{A} = \{0,\dots,d-1\}$, we obtain
\begin{equation}\label{eqUpperBoundasaminimum}
    h_{PS}(X) \leq \min \{\log d,\ (\log d + \log \|A_1\|_{op} + \dots + \log \|A_k\|_{op})/k\}.
\end{equation}

\begin{example}
Consider $ X $ the tree-shift given by the matrices
$$
A_1 = \begin{pmatrix}
1 & 1 \\
1 & 1
\end{pmatrix}, \ \ \
A_2, \dots, A_5 = \begin{pmatrix}
1 & 0 \\
0 & 1
\end{pmatrix}.
$$
By Proposition \ref{PropUpperBondEntropyPS},
$$
h_{PS}(X) \leq \dfrac{\log 2 + \log 2 + 4\log 1}{5} = \dfrac{2}{5} \log 2 < \log 2.
$$
This shows that, in this case, Proposition \ref{PropUpperBondEntropyPS} provides a better upper bound for the entropy of $X$ than Proposition \ref{PositivePSentropyImpliesBCTotalEntropy}. However, if all $A_1, \dots, A_5$ had norm $2$, the upper bound given by Proposition \ref{PropUpperBondEntropyPS} would be $ \frac{6}{5} \log 2 $, so, in this case, Proposition \ref{PositivePSentropyImpliesBCTotalEntropy} is more suitable.
\end{example}

Inspired by what is observed in the one-dimensional shift space, one could investigate whether there exists an upper bound for $h_{PS}$ of a tree-shift in terms of the logarithm of the spectral radius of its matrices (corresponding to the entropy of the one-dimensional shift space determined by each matrix). Neither taking the mean of these values nor the minimum (or maximum) work for tree-shifts. For example, we show in Section \ref{sectionentropytable} that the entropy of $X_{11} = (B,E)$ is approximately $0.23435$, while the eigenvalues of $B$ and $E$ are all equal to one. Moreover, the tree-shift in Example \ref{ExampleZerohPS-PositivehBC} has $h_{PS} = 0$ but the maximum of the eigenvalues of $A_1$ and $A_2$ is $\log (1+\sqrt{5})/2$.

Now, let $M, N$ be $d \times d $ transition matrices. We write $M \succcurlyeq N$ if $M_{ij} \geq N_{ij}$ for all $1 \leq i,j \leq d$.

\begin{proposition}\label{EntropyMatrixBiggerEntries}
Let $X = (A_1,A_2)$ and $Y = (A_1,A_3)$ be binary tree-shifts. If $ A_2 \succcurlyeq A_3 $, then $ h_{PS}(X) \geq h_{PS}(Y) $ and $ h_{BC}(X) \geq h_{BC}(Y) $.
\end{proposition}

\begin{proof}
In fact, $ f_X(\mymathbb{1})_i \geq f_Y(\mymathbb{1})_i $ for all $ 1 \leq i \leq d$ by definition, and the inequality is preserved by the iterates of these functions. Consequently,
$$
h_{PS}(X) = \lim\limits_{n\rightarrow\infty} \dfrac{\log \|f_X^n(\mymathbb{1})\|}{1+k+\dots+k^n}
\geq
\lim\limits_{n\rightarrow\infty} \dfrac{\log \|f_Y^n(\mymathbb{1})\|}{1+k+\dots+k^n} = h_{PS}(Y).
$$
Similarly for $h_{BC}$.
\end{proof}

\begin{remark}\label{RemarkComparingEntropiesRestrictingTransitions}
One can prove that, if $ X = (A_1,\dots,A_k) $ and $ Y = (B_1,\dots,B_k) $ are tree-shifts with $ d \times d $ matrices such that $ A_i \succcurlyeq B_i $ for all $ i = 1, \dots, k $, then $ h_{PS}(X) \geq h_{PS}(Y) $. In particular, if $ h_{PS}(Y) > 0 $, we can guarantee that $ h_{PS}(X) > 0 $.
\end{remark}

\begin{remark}
Suppose that $A_i \succcurlyeq M$ for some irreducible $ M $, and define $Z$ the tree-shift whose all $k$ transitions are given by $M$. Then, by \cite{PetersenSalama1,PetersenSalama2} and our previous argument, we get
$$
h_{PS}(X) \geq h_{PS}(Z) \geq h_{top}(\Sigma_M),
$$
where $\Sigma_M$ is the one-dimensional Markov subshift defined by $M$.
\end{remark}

In \cite{BanChangEtAll2022} the authors proved that, if a tree-shift has a single transition matrix for all directions with norm at least two, its entropy is positive. In the case that the allowed transitions of a tree-shift are given by different matrices, this property (norm at least two) being satisfied by all matrices is not sufficient to guarantee positive $h_{PS}$, as can be observed in Example \ref{ExampleZerohPS-PositivehBC}. Moreover, the authors proved that a tree-shift with the same transition matrix for all directions having norm one has zero entropy $h_{PS}$ (and also $h_{BC}$, as a consequence of their proof). This result can be generalized for tree-shifts given by different transition matrices in the following way:

\begin{proposition}\label{PropFiniteTreeShiftsZeroEntropy}
Let $ X = (A_1, \dots, A_k)$ be a tree-shift over the alphabet $ \mathcal{A} = \{0,\dots,d-1\} $ and suppose that $\|A_1\|_{op} = \dots = \|A_k\|_{op} = 1$. Then $ h_{PS}(X) = h_{BC}(X) = 0 $.
\end{proposition}

\begin{proof}
The proof functions on account of the fact that $ p(n) = d $ for all $n \geq 1$.
\end{proof}

Now, we find a lower bound for the entropy of a Markov tree-shift when one of its transition matrices is aperiodic to prove that this tree-shift has positive entropy.

\begin{proposition}\label{PropositionAperiodicMatrixImpliesPositiveEntropy}
Let $X = (A_1, \dots, A_k)$ be a tree-shift such that $A_1$ is aperiodic. Then $h_{PS}(X) > 0$.
\end{proposition}

\begin{proof}
Let $\ell$ be a row of $A_1$ with a row sum of $\gamma \geq 2$, by Remark \ref{RemarkAperiodicTransitionMatrix}. Since $A_1$ is aperiodic, there exists $n_0 \in \N$ such that $A_1^n(i,j) > 0$ for all $n \geq n_0$ and $i,j \in \{1, \dots, d\}$. We want to count the number of nodes of the form $x a_1^{n_0-1}$, $x \in \Sigma^*$, at the last level of a block of length $n$. The expression $a_1^{n_0-1}$ denotes the concatenation of $a_1$ $n_0-1$ times. These nodes allow $t_x t_{x a_1} t_{x a_1^2} \dots t_{x a_1^{n_0-1}}$ to be a word of length $n_0$, constructed using only the matrix $A_1$, and satisfying $t_{xa_1^{n_0}-1} = \ell$, regardless of $t_x$. Then, this word can be expanded to a word of length $n_0 + 1$ in $\gamma$ different ways.

A block of length $n_0 - 1$ has only the node $a_1^{n_0-1}$ on its last level with the desired property (in this case, $t_x = t_\epsilon$), and smaller blocks do not need to be considered. A block of length $n_0$ has $k$ nodes of the form $x_1 a_1^{n_0-1}$, $x_1 \in \Sigma^1 \setminus \{ \epsilon \}$; a block of length $n_0 + 1$ has $k^2$ nodes of the form $x_2 a_1^{n_0-1}$, $x_2 \in \Sigma^2 \setminus \Sigma^1$. By this argument, one can prove that a block of length $n_0 + q - 1$ has exactly $k^q$ nodes of the form $x a_1^{n_0 - 1}$ on its last level.

Therefore, since in a block of length $n_0 + q - 1$ we can fit $k^q$ words $t_x t_{x a_1} t_{x a_1^2} \dots t_{x a_1^{n_0 - 1}}$ independently, we obtain $p(n_0 + q) \geq d \cdot \gamma^{k^q}$. Thus,
$$
h_{PS}(X) = \lim\limits_{q\rightarrow\infty} \dfrac{k-1}{k^{n_0 + q + 1}} \log p(n_0 + q) \geq \dfrac{k-1}{k^{n_0 + 1}} \log \gamma.
$$
We then prove that the entropy $h_{PS}$ of $X$ is strictly positive and provided a lower bound for it.
\end{proof}

Using a similar argument as in Proposition \ref{PropositionAperiodicMatrixImpliesPositiveEntropy}, we can also prove an analogous result assuming a weaker hypothesis:

\begin{proposition}\label{PropositionIrreducibleMatrixNorm2ImpliesPositiveEntropy}
Let $X = (A_1, \dots, A_k)$ be a tree-shift such that $A_1$ is a $d \times d$ irreducible matrix with $\| A_1\|_{op} = \gamma \geq 2$. Then $h_{PS}(X) > 0$.
\end{proposition}

\begin{proof}
Let $1 \leq \ell \leq d$ be a row of $A_1$ with a sum equal to $\gamma$ and fix $i \in \mathcal{A}$ arbitrarily. Since $A_1$ is irreducible, $A_1^d(i, \ell) > 0$ or $A_1^d(i,\ell) = 0$ but $A_1^m(i, \ell) > 0$ for some $m < d$. Let us consider both cases separately.

If $A_1^d(i, \ell) > 0$, there exists at least one admissible block of length $d-1$ in $X$ with root $i$ and $\ell$ at the node $a_1^{d-1}$. Then, this block can be extended to a block of length $d$ in $\gamma$ different ways, at minimum. In the case that $A_1^m(i, \ell) > 0$ for $m < d$, there is at least one admissible block of length $m-1$ with root $i$ and $\ell$ at node $a_1^{m-1}$ and, consequently, no less than $\gamma$ different extensions to a block of length $d$ exist.

Considering both cases previously presented for any $i \in \mathcal{A}$, we can take $d \cdot \gamma$ as a conservative estimate of the number of allowed blocks of length $d$ in $X$. Let us continue with this argument.

A block of length $d+1$ can be seen as $k$ blocks of length $d$ attached at its root, then, for each root $j \in \mathcal{A}$, at least $\gamma$ blocks of length $d$ that can be attached in each direction exist. This implies that $d \cdot \gamma^k$ is a lower bound for the number of admissible blocks in $X$ with length $d+1$. Following this estimate, for any $n \geq d$ one can show that $p(n) \geq d \cdot \gamma ^{k^{n-d}}$. Then,
$$
h_{PS}(X) \geq \lim\limits_{n\rightarrow\infty} \dfrac{k-1}{k^{n+1}} \log \left( d \cdot \gamma^{k^{n-d}} \right) = \dfrac{k-1}{k^{d+1}} \log \gamma.
$$
\end{proof}

An aperiodic matrix is also irreducible, so we get by Propositions \ref{PropositionAperiodicMatrixImpliesPositiveEntropy} and \ref{PropositionIrreducibleMatrixNorm2ImpliesPositiveEntropy} two lower bounds (that can possibly be different) to the entropy $h_{PS}$ of a Markov tree-shift given by $k$ matrices such that one of them is aperiodic. Since, in all cases, the power of an aperiodic $d \times d$ matrix that makes all of its entries strictly positive is less or equal to $d$, the lower bound provided by Proposition \ref{PropositionAperiodicMatrixImpliesPositiveEntropy} can be less than the one obtained using Proposition \ref{PropositionIrreducibleMatrixNorm2ImpliesPositiveEntropy}, as one should expect.

To establish another class of Markov tree-shifts with positive entropy (no matrix is irreducible), we need the following auxiliary result.

\begin{lemma}\label{LemmaPositiveEntropyIdentityMatrices}
    Let $M$ be a $2\times2$ matrix with $\|M\|_{op} = 2$ and $I$ be the $2\times2$ identity matrix. Then, for any $k \geq 2$, the tree-shift $X = (M, I, \dots, I)$ consisting of $k$-trees has a positive $h_{PS}$.
\end{lemma}

\begin{proof}
If $M$ is aperiodic, Proposition \ref{PropositionAperiodicMatrixImpliesPositiveEntropy} yields $h_{PS}(X) > 0$. Still to be considered is $M$ as one of the following matrices:
$$
    \begin{pmatrix}
        1 & 1 \\
        0 & 1
    \end{pmatrix}
    \ \ \ \ \text{ and } \ \ \ \
    \begin{pmatrix}
        1 & 0 \\
        1 & 1
    \end{pmatrix}.
$$
Let us take into account the first case, since both matrices are the same transition matrix up to a correspondence between their alphabet (notice that, in general, considering the second matrix does not generate an equivalent tree-shift as addressing the first matrix, but here we are restricted to the particular case where all the other transition matrices are the identity).

Therefore, take $X = (M, I, \dots, I)$, where $M$ is the left matrix presented above. We have $f(x_1, x_2) = (x_1^{k-1}(x_1 + x_2), x_2^k)$ and, for $\eta = (\eta_1, 1)$ with $\eta_1 = x_1/x_2$, we obtain $N(\eta) = 1$ and $F(\eta) = (\eta_1^{k-1}(\eta_1 + 1), 1)$. It is possible to estimate $p(n)$ as
    $$
    \| f^n(\mymathbb{1})\| = F^n(\mymathbb{1})_1 +1 > 2^{k^{n-1}} + 1,
    $$
so we get
$$
h_{PS}(X) \geq \lim\limits_{n\rightarrow\infty} \dfrac{k-1}{k^{n+1}-1} \log 2^{k^{n-1}} = \frac{k-1}{k^2} \log 2.
$$
Then the result follows.
\end{proof}

\begin{remark}\label{RemarkRestrinctionToAnSubAlphabet}
    If $X$ is a tree-shift over an alphabet $\mathcal{A}$ and $\tilde{\mathcal{A}} \subset \mathcal{A}$, define $X|_{\tilde{\mathcal{A}}}$ the subset of $k$-trees $t$ of $X$ such that $t_x \in \tilde{\mathcal{A}}$ for all $x \in \Sigma^*$. Then, $X|_{\tilde{\mathcal{A}}}$ is a tree-subshift and $h_{PS}(X) \geq h_{PS} \left(X|_{\tilde{\mathcal{A}}} \right)$.
\end{remark}

\begin{proposition}
Let $X = (A_1, \dots, A_k)$ be a tree-shift over $\mathcal{A} = \{0, \dots, d-1\}$. Suppose that there exists $1 \leq \ell \leq d-1$ such that
$$
\tilde{A}_1 = 
\begin{pmatrix}
A_1(\ell, \ell) & A_1 (\ell, \ell+1) \\
A_1(\ell + 1, \ell) & A_1(\ell+1, \ell+1)
\end{pmatrix}
$$
satisfies $\| \tilde{A}_1\| \geq 2$. Define, in a similar way, 
$$
\tilde{A}_m = 
\begin{pmatrix}
A_m(\ell, \ell) & A_m (\ell, \ell+1) \\
A_m(\ell + 1, \ell) & A_m(\ell+1, \ell+1)
\end{pmatrix}
$$
for all $1 < m \leq k$. If $\tilde{A}_m \succcurlyeq I$ for all $m$, where $I$ is the $2 \times 2$ identity matrix, then $h_{PS}(X) > 0$.
\end{proposition}

\begin{proof}
Define $Y = (\tilde{A}_1, I, \dots, I)$ and $\tilde{Y} = (\tilde{A}_1, \tilde{A}_2, \dots, \tilde{A}_k)$ tree-shifts consisting of $k$-trees. By Lemma \ref{LemmaPositiveEntropyIdentityMatrices} we know that $h_{PS}(Y) > 0$ and, by Remark \ref{RemarkComparingEntropiesRestrictingTransitions}, we obtain that $h_{PS}(\tilde{Y}) \geq h_{PS}(Y)$. Using Remark \ref{RemarkRestrinctionToAnSubAlphabet}, we have $h_{PS}(X) \geq h_{PS}(\tilde{Y})$. It ends the proof.
\end{proof}

\section{The entropy of all binary tree-shifts over the alphabet $\{0,1\}$}\label{sectionentropytable}

Inspired by the investigation of several examples exposed in \cite{PetersenSalama1,PetersenSalama2}, in this section we use the expression of $f^n(\mymathbb{1})$, given in equation \eqref{eqdefinitionfunctionf}, and Proposition \ref{Prop-p(n)=f^n} to calculate the entropy of some examples, namely all binary tree-shifts over the alphabet $ \mathcal{A} = \{0,1\} $ whose matrices has no row or column with all entries equal to zero. The remaining cases have trivially zero entropy. Using Propositions \ref{PositivePSentropyImpliesBCTotalEntropy} and \ref{PropFiniteTreeShiftsZeroEntropy} we can easily find $ h_{BC} $ for each case.

Consider matrices $ A - G $ defined as follows: 
$$
A=\begin{pmatrix}
1 & 1 \\
1 & 1
\end{pmatrix}
\ \ \ \
B=\begin{pmatrix}
1 & 0 \\
0 & 1
\end{pmatrix}
\ \ \ \
C=\begin{pmatrix}
0 & 1 \\
1 & 0
\end{pmatrix}
\ \ \ \
D=\begin{pmatrix}
1 & 1 \\
1 & 0
\end{pmatrix}
$$
$$
E=\begin{pmatrix}
1 & 0 \\
1 & 1
\end{pmatrix}
\ \ \ \
F=\begin{pmatrix}
0 & 1 \\
1 & 1
\end{pmatrix}
\ \ \ \
G=\begin{pmatrix}
1 & 1 \\
0 & 1
\end{pmatrix}
$$

We also establish the following notation:
$$
\begin{array}{ccccc}
    X_1=(A,A), & X_2=(A,B), & X_3=(A,C), & X_4=(A,D), & X_5=(A,E), \\
    X_6=(A,F), &
    X_7=(A,G), & X_8=(B,B), & X_9=(B,C), & X_{10}=(B,D), \\
    X_{11}=(B,E), & X_{12}=(B,F), &
    X_{13}=(B,G), & X_{14}=(C,C), & X_{15}=(C,D), \\ X_{16}=(C,E), & X_{17}=(C,F), & X_{18}=(C,G), &
    X_{19}=(D,D), & X_{20}=(D,E), \\ X_{21}=(D,F), & X_{22}=(D,G), & X_{23}=(E,E), & X_{24}=(E,F), &
    X_{25}=(E,G), \\
    &
    X_{26}=(F,F), & X_{27}=(F,G), & X_{28}=(G,G), &  \\
\end{array}
$$

Let us present explicitly the functions $ f $, $ F $ and $ N $ for $ k = d = 2 $. For the binary tree-shift $ X = (P,Q) $, we have
$$
f(x) = (Px) * (Qx)
$$
and, using the coordinates $\eta_i = \frac{x_i}{x_2}$, $i = 1,2$, $x_2 \neq 0$, we get
$$
N(\eta)=(P_{21}\eta_1+P_{22}) (Q_{21}\eta_1+Q_{22})
$$
and
$$
  F(\eta)_i = \frac{ f(x)_i }{f(x)_2} = \frac{(P_{i1}\eta_1+P_{i2}) (Q_{i1}\eta_1+Q_{i2}) }
  {N(\eta)}.
$$

It is equivalent to consider the new coordinates $\tilde{\eta}_i=x_i/x_1$, $ i=1,2 $, $x_1 \neq 0$, with correspondent functions
$$
\tilde{N}(\tilde{\eta}) = (P_{11}+P_{12}\tilde{\eta}_2) (Q_{11}+Q_{12}\tilde{\eta}_2)
$$
and
$$
  \tilde{F}(\tilde{\eta})_i = \frac{ f(x)_i }{f(x)_1} = \frac{(P_{i1}+P_{i2}\tilde{\eta}_i) (Q_{i1}+Q_{i2}\tilde{\eta}_i) }
  {\tilde{N}(\tilde{\eta})}.
$$

Hence
$$
\begin{array}{rl}
     f^n \left(
      \left(
     \begin{array}{c}
       1  \\
        1
     \end{array}
      \right)
    \right)  = & N(\mymathbb{1})^{2^{n-1}}N(F(\mymathbb{1}))^{2^{n-2}} N(F^2(\mymathbb{1}))^{2^{n-3}} \dots N(F^{n-1}(\mymathbb{1}))
     \left(
     \begin{array}{c}
       F^n(\mymathbb{1})_1   \\
        1
     \end{array}
      \right) \\
    = & \tilde{N}(\mymathbb{1})^{2^{n-1}}\tilde{N}( \tilde{F}(\mymathbb{1}))^{2^{n-2}} \tilde{N}(\tilde{F}^2(\mymathbb{1}))^{2^{n-3}} \dots \tilde{N}(\tilde{F}^{n-1}(\mymathbb{1}))
     \left(
     \begin{array}{c}
        1   \\
        \tilde{F}^n(\mymathbb{1})_2
     \end{array}
      \right).
\end{array}
$$

\begin{remark}\label{RemarkCoordinatesChaangeforEta_i}
If $X$ and $Y$ are two tree-shifts such that the functions $F$ and $N$ corresponding to $X$ coincide with $F$ and $N$ or with $\tilde{F}$ and $\tilde{N}$ corresponding to $Y$, by Remark \ref{RemarkCoordinatesChaange} these tree-shifts have the same number of allowed blocks of length $n$, for each $n \in \N$. Moreover, since $d = 2$ these tree-shifts are conjugated.
\end{remark}

In what follows, we omit the subindex of $ \eta_i $, writing $ \eta $ when referring to both $ \eta_1 $ and $ (\eta_1,1) $ and, in a similar way, we write $ 1 $ for $ (1,1) $. It should not cause any confusion when considering the context. Moreover, since $ F (\eta) = (F(\eta)_1,1) $, we also refer to $F(\eta)_1$ simply as $F(\eta)$. Similarly with $\tilde{\eta}$ and $\tilde{F}(\tilde{\eta})$. Also, remember that we denote $ (\Fib(n))_{n \geq 1} = (1,1,2,3,5,8,\dots) $ the Fibonacci sequence.

In Proposition \ref{PositivePSentropyImpliesBCTotalEntropy} we proved that $ h_{PS}(X_1) = \log 2 $.
The remaining cases are studied below.

\subsection{Case 1.} $ X_2 = (A,B) $ and $ X_3 = (A,C) $

For the tree-shift $ X_2 $, we have $ F(\eta) = \eta $ and $ N(\eta) = \eta + 1 $, thus $ N(F^j(1)) = 2 $ for all $j \geq 1$. It implies that
$$
\|f^n(1)\| = 2^{2^{n-1}} 2^{2^{n-2}} \dots 2(1+1),
$$
and, consequently,
$ h_{PS}(X_2) = \frac{1}{2} \log 2 $. A straightforward calculation shows that the functions $F$ and $N$ corresponding to $X_3$ satisfies $N(1) = N(F^j(1)) = 2$ for all $j \geq 1$. Therefore, $X_2$ and $X_3$ have the same entropy.

\subsection{Case 2.} $ X_4 = (A,D) $ and $ X_6 = (A,F) $

For $ X_4 $ we have $ F(\eta) = 1 + 1/\eta $ and $ N(\eta) = \eta (\eta + 1) $, and it yields
$$
F^j(1) = \dfrac{\Fib(j+2)}{\Fib(j+1)}\ \ \ \text{and}\ \ \ N(F^j(1)) = \dfrac{\Fib(j+3)\Fib(j+2)}{\Fib(j+1)^2},\ \ \ j \geq 0.
$$
After some simplifications, we obtain, for $ n \geq 3 $,
$$
\|f^n(1)\| =
\Fib(3)^{2^{n-1}}
\Fib(4)^{2^{n-2}}
\Fib(5)^{2^{n-3}}
\dots
\Fib(n)^{2^2}
\Fib(n+1)^2
\Fib(n+2)
\Fib(n+3).
$$
Therefore,
$$
h_{PS}(X_4) = \lim\limits_{n\rightarrow\infty} \left( \displaystyle\sum\limits_{j = 2}^{n+1} \dfrac{1}{2^j} \log \Fib(j + 1) + \dfrac{1}{2^{n+1}} \log \Fib(n+3) \right) = \displaystyle\sum\limits_{n=2}^\infty\dfrac{1}{2^n}\log\Fib(n+1).
$$

We were not able to explicitly determine the number $ h_{PS}(X_4) $, however, it is possible to present some upper and lower bounds for $ h_{PS} $ taking advantage of the fact that we can easily calculate many terms of its series. For the lower bound, we consider
$$
h_{PS}(X_4) \geq \displaystyle\sum\limits_{n=2}^5 \dfrac{1}{2^n} \log \Fib(n+1) \approx 0.47619 \approx \log 1.6099.
$$
Moreover, since $ \sum_{n=1}^\infty \frac{n}{2^n} = 2 $ and $ \Fib(n) < 2^{n-3} $ for all $n \geq 7$, we get, for an upper bound,
$$
\begin{array}{rl}
    h_{PS}(X_4) = & \displaystyle\sum\limits_{n=2}^5 \dfrac{1}{2^n} \log \Fib(n+1) + \displaystyle\sum\limits_{n=6}^\infty \dfrac{1}{2^n} \log \Fib(n+1) \\
    \leq & \displaystyle\sum\limits_{n=2}^5 \dfrac{1}{2^n} \log \Fib(n+1) + \displaystyle\sum\limits_{n=6}^\infty \dfrac{1}{2^n} \log 2^{n-2}\\
    < & 0.47622 + \dfrac{1}{4} \log 2 \displaystyle\sum\limits_{n=4}^\infty \dfrac{n}{2^n} \leq 0.58452 \approx \log 1.79414.
\end{array}
$$
These bounds can naturally be improved by considering more terms for the lower bound, which was not our aim here. The estimates were taken considering the natural logarithm.

By Remark \ref{RemarkCoordinatesChaangeforEta_i}, $X_4$ and $X_6$ have the same entropy.

\subsection{Case 3.} $ X_5 = (A,E) $ and $ X_7 = (A,G)$

For $ X_5 $ we get $ \tilde{F}^j(1) = j+1 $ and $ \tilde{N}(\tilde{F}^j(1)) = j+2$, $ j \geq 1 $. Then,
$$
\|f^n(1)\| = 2^{2^{n-1}} 3^{2^{n-2}} \dots n^2 (n+1) (n+2),
$$
and, since $ \lim_{x\rightarrow\infty} \log(x+1)/2^x = 0 $,
$$
h_{PS}(X_5)
= \lim\limits_{n \rightarrow \infty} \left( \displaystyle\sum\limits_{j = 2}^{n+1} \dfrac{1}{2^j} \log j + \dfrac{1}{2^{n+1}} \log (n+2) \right)
= \displaystyle\sum\limits_{n=2}^\infty \dfrac{1}{2^n} \log n.
$$

For an approximation of this value, we proceed in a similar way as in the previous case. A particular choice of the authors of this study shows that
$$
h_{PS}(X_5) \geq \displaystyle\sum\limits_{n=2}^5 \dfrac{1}{2^n} \log n \approx 0.44755 \approx \log 1.56448
$$
and
$$
\begin{array}{rl}
    h_{PS}(X_5) = & \displaystyle\sum\limits_{n=2}^5 \dfrac{1}{2^n} \log n + \displaystyle\sum\limits_{n=6}^\infty \dfrac{1}{2^n} \log n \leq \displaystyle\sum\limits_{n=2}^5 \dfrac{1}{2^n} \log n + \displaystyle\sum\limits_{n=6}^\infty \dfrac{n}{2^n} \\
    \approx & 0.44755 + \frac{7}{64} = 0.55692 \approx \log 1.7453.
\end{array}
$$

Tree-shifts $ X_5 $ and $ X_7 $ have the same entropy by Remark \ref{RemarkCoordinatesChaangeforEta_i}.

\subsection{Case 4.} $ X_8 = (B,B) $, $ X_9 = (B,C) $ and $ X_{14} = (C,C) $

In any case we get $ F^j(1) = 1 $ and $ N(F^j(1)) = 1 $ for all $ j \geq 1 $. Therefore, the entropy $ h_{PS} $ of any of these systems is zero.

\subsection{Case 5.} $ X_{10} = (B,D) $ and $ X_{12} = (B,F) $

For $ X_{10} $ we have $ F^j(1) = j+1 $ and $ N(F^j(1)) = j+1 $. Then,
$$
\|f^n(1)\| = 2^{2^{n-2}} 3^{2^{n-3}} \dots n(n+2),
$$
and Case 3 gives $ h_{PS}(X_{10}) = \frac{1}{2} h_{PS}(X_{5})$. Remark \ref{RemarkCoordinatesChaangeforEta_i} guarantees that $ h_{PS}(X_{10}) = h_{PS}(X_{12}) $.

\subsection{Case 6.} $ X_{11} = (B,E) $ and $ X_{13} = (B,G) $

Consider the tree-shift $X_{11}$. We get $ \tilde{N}(\tilde{\eta}) = 1 $ and, for any $ j\geq1 $, $ \tilde{F}^j(1) = c_j $, where $ c_1 = 2 $ and $ c_j = c_{j-1} (c_{j-1} + 1) $. Then
$$
\|f^n(1)\| = c_n + 1
$$
and
$$
h_{PS}(X_{11}) = \lim\limits_{n\rightarrow\infty} \dfrac{1}{2^{n+1}} \log (c_n + 1) := \xi.
$$

We could not determine the numeric value of $ h_{PS}(X_{11}) $ from the limit above. However, we use the same approach as Petersen and Salama in \cite{PetersenSalama1}. Consulting the sequence $ (c_n)_{n\geq1} = (2,6,42,1806,\dots) $ on the Online Encyclopedia of Integer Sequences (OEIS) \cite{OEIS}, we obtain Sequence A007018, that states that $ c_n $ is the integer directly below the real number $ \theta^{2^n}-1/2 $, where $ \theta \approx 1.59791 $. We now prove that $ h_{PS}(X_{11}) = \frac{1}{2} \log\theta $, and this value is approximately $ 0.23435 $.

Indeed, since
$
c_n - (\theta^{2^n}-\frac{1}{2}) = \delta_n < 1$ for each $ n \geq 1 $, we have
$$
\log( c_n+1 ) = \log \left( \theta^{2^n} + \frac{1}{2} + \delta_n \right).
$$
Then,
$$
\lim\limits_{n\rightarrow\infty} \dfrac{1}{2^{n+1}} \left(\log (\theta^{2^n} + \frac{1}{2} + \delta_n) - \log \theta^{2^n} \right) = \lim\limits_{n\rightarrow\infty} \dfrac{1}{2^{n+1}} \left( \log \dfrac{\theta^{2^n} + \frac{1}{2} + \delta_n}{\theta^{2^n}} \right) = 0,
$$
which implies that
$$
h_{PS}(X_{11}) =
\lim\limits_{n\rightarrow\infty} \dfrac{1}{2^{n+1}} \log \left( \theta^{2^n} + \frac{1}{2} + \delta_n \right)
= \lim\limits_{n\rightarrow\infty} \dfrac{1}{2^{n+1}} \log \theta^{2^n}
= \dfrac{1}{2} \log \theta.
$$

We use Remark \ref{RemarkCoordinatesChaangeforEta_i} to guarantee that the entropy of $ X_{11} $ and $ X_{13} $ is the same.

\subsection{Case 7.} $ X_{15} = (C,D) $ and $ X_{17} = (C,F) $

For the tree-shift $ X_{15} $ we obtain $ N(\eta) = \eta^2 $ and $ F^j(1) = e_j/e_{j-1}^2 $, where $ e_0 = 1 $, $ e_1 = 2 $ and $ e_j = e_{j-2}^2 (e_{j-2}^2 + e_{j-1}) $. Then,
$$
\|f^n(1)\| = e_n + e_{n-1}^2
$$
and, as a consequence,
$$
h_{PS}(X_{15}) = \lim\limits_{n\rightarrow\infty} \dfrac{1}{2^{n+1}} \log (e_n + e_{n-1}^2) := \alpha.
$$
By the time of the consult, we could not find a reference on OEIS for either sequence, not for $ (e_n)_{n\geq0} = (1,2,3,28,333,875728,\dots)$ nor for $ (\|f^n(1)\|)_{n\geq1} = (3,7,37,1117,986617,\dots)$ in order to obtain a good approximation for the entropy of $X_{15}$. However, we can estimate an upper bound to $ h_{PS} $ better than $ \log 2 $, given by equation \eqref{eqUpperBoundasaminimum}, using the number of allowed blocks of length $ 2 $ as follows.

Considering the tree-shift $X_{15}$, each configuration on the last level of the allowed blocks of length $ 1 $ appears only once. It suggests that the configuration of the last level of a block of any length totally determines the tree-shift. However, more restrictions appear in regard to blocks of length $n \geq 2$, what can be seen by $\|f^2(1)\| = 7$, $\|f^3(1)\| = 37$ and so on. We consider the information of the existence of $37$ different configurations for the last level of a block of length $3$. For the reason that one can choose $2^{n-3}$ blocks of length $3$ for the last line of a block of length $n \geq 3$, we see that $ 37^{2^{n-3}} $ is an upper bound to the number of allowed blocks of length $ n \geq 3 $. Therefore,
$$
h_{PS}(X_{15}) \leq \lim\limits_{n\rightarrow\infty} \dfrac{1}{2^{n+1}} \log 37^{2^{n-3}} = \dfrac{1}{16} \log 37 \approx \log 1.25318.
$$
The reader can easily find a more refined upper bound for the entropy using $ \|f^m(1)\| $ for any $m > 2$ if needed.

Subsequently, in order to prove that this entropy is positive, we present a (nonoptimal) lower bound. We can see that $ \| f^n(1) \| > 2^{2^{n-1}} $ for $n = 1, \dots, 4$. Supposing that the same property holds for some $i \geq 4$, we find
$$
\|f^{i+1}(1)\| = e_{i+1} + e_i^2 > e_{i+1} = e_{i-1}^2 (e_{i-1}^2 + e_i) > 2^{2^{i-1}} e_{i-1}^2 > 2^{2^{i-1}} e_{i-1} > 2^{2^{i-1}}2^2 = 2^{2^i},
$$
so $ \| f^n(1)\| > 2^{2^{n-1}} $ for all $ n \geq 1 $. Consequently,
$$
h_{PS}(X_{15}) \geq \lim\limits_{n\rightarrow\infty} \dfrac{1}{2^{n+1}} \log 2^{2^{n-1}} = \dfrac{1}{4} \log 2.
$$
We then can conclude that $ \frac{1}{4} \log2 < \alpha \leq \frac{1}{16} \log 37 $, or yet, $ 0.17329 \leq \alpha \leq 0.22568 $.

Once more, Remark \ref{RemarkCoordinatesChaangeforEta_i} guarantees that the entropy of $ X_{15} $ and $ X_{17} $ coincide with each other.

\subsection{Case 8.} $ X_{16} = (C,E) $ and $ X_{18} = (C,G) $

Regarding to the tree-shift $ X_{16} $, for each $ j \geq 1 $ we get $ F^j(1) = \Fib(j+1) / \Fib(j+2) $ and $ N(F^j(1)) = \Fib(j+1)\Fib(j+3) / \Fib(j+2)^2 $. Therefore,
$$
\|f^n(1)\| =
\Fib(3)^{2^{n-3}}
\Fib(4)^{2^{n-4}}
\dots
\Fib(n)
\Fib(n+3).
$$
We then get
$$
h_{PS}(X_{16}) = \dfrac{1}{4} \displaystyle\sum\limits_{n=2}^\infty \dfrac{1}{2^n} \log \Fib(n+1) = \dfrac{1}{4} h_{PS}(X_4).
$$
By Case 2, we obtain $ 0.11903 \leq h_{PS}(X_{16}) \leq 0.14615$, or yet, $\log 1.1264 \leq h_{PS}(X_{16}) \leq \log 1.15737$, and, by Remark \ref{RemarkCoordinatesChaangeforEta_i}, $ X_{16} $ and $ X_{18} $ have the same entropy.

\subsection{Case 9.} $ X_{19} = (D,D) $ and $ X_{26} = (F,F) $

In regards to $X_{19}$, we obtain $ N(\eta) = \eta^2 $ and $ F^j(1) = r_j/r_{j-1}^2 $, where $ r_0 = 1, r_1 = 2^2, r_j = (r_{j-1} + r_{j-2}^2)^2$. Therefore,
$$
\|f^n(1)\| = r_n + r_{n-1}^2.
$$
Thus,
$$
h_{PS}(X_{19}) =
\lim\limits_{n\rightarrow\infty} \dfrac{1}{2^{n+2}} \log (r_n + r_{n-1}^2)^2 =
\lim\limits_{n\rightarrow\infty} \dfrac{1}{2^{n+1}} \log r_n.
$$

Using OEIS, Petersen and Salama \cite{PetersenSalama1} proved that $h_{PS}(X_{19})$ is approximately $0.509$. Using Remark \ref{RemarkCoordinatesChaangeforEta_i}, this results in the same entropy of $X_{26}$.

\subsection{Case 10.} $ X_{23} = (E,E) $ and $ X_{28} = (G,G) $

For $ X_{28} = (G,G) $ we obtain $F(\eta) = (\eta + 1)^2$ and $N(\eta) = 1$. Then, for all $n \in \N$,
$$
\| f^n(1)\| = s_n + 1,
$$
where $s_1 = 4$ and $s_n = (s_{n-1} + 1)^2$, $n > 1$. The sequence $(s_n)_{n \geq 1} = (4, 25, 676, 458329, \dots)$ corresponds to Sequence A004019 from OEIS. From this encyclopedia we understand that $s_n$ is the integer closer to $\ell^{2^n} - 1$, where $\ell \approx 2.258518$. According to arguments similar to the ones used in Case 6, we obtain
$$
h_{PS} (X_{28}) =
\lim\limits_{n \rightarrow \infty} \dfrac{1}{2^{n+1}} \log (s_n + 1) =
\lim\limits_{n \rightarrow \infty} \dfrac{1}{2^{n+1}} \log \ell^{2^n} \approx
\dfrac{1}{2} \log 2.25852
\approx 0.40735.
$$
Once more, Remark \ref{RemarkCoordinatesChaangeforEta_i} guarantees that the entropy of $X_{23}$ and $X_{28}$ are the same.

\subsection{Case 11.} $ X_{20} = (D,E) $, $ X_{21} = (D,F) $, $ X_{25} = (E,G) $ and $ X_{27} = (F,G) $

In any of these cases, we obtain $ N(F^j(1)) = 2 $ for all $ j \geq 0 $, thus,
$$
\|f^n(1)\| = 2^{2^{n-1}}
2^{2^{n-2}}
\dots
2 (1+1),
$$
and, for $ \ell \in \{20,21,25,27\} $,
$$
h_{PS}(X_\ell) = \log2 \displaystyle\sum\limits_{n=2}^\infty \dfrac{1}{2^n} = \dfrac{1}{2} \log2. $$

\subsection{Case 12.} $ X_{22} = (D,G) $ and $ X_{24} = (E,F) $

Considering $ X_{22} $, we get $ N(\eta) = \eta $ and $ F(\eta) = (1 + \eta) (1 + 1/\eta) $, thus $ F^j(1) = u_j/ \prod_{i=1}^{j-1}u_i $, where $ u_0 = 1 $, $ u_1 = 4 $ and $ u_j = \left (u_{j-1} + \prod_{i=1}^{j-2} u_i \right)^2 $ for all $ j \geq 2 $. We then obtain
$$
\|f^n(1)\| = u_n + \prod\limits_{i=1}^{n-1} u_i,
$$
so $ h_{PS}(X_{22}) = \lim\limits_{n\rightarrow\infty} \dfrac{1}{2^{n+1}} \log u_n $.

We could not find any reference in OEIS to the sequence $(u_n)_{n\geq0} = (1, 4, 25, 841, 885481, \dots)$, so we use a similar strategy as seen in Case 7. We have $ (\|f^n(1)\|)_{n \geq 1} = (5, 29, 941, 893891, \dots) $, and estimating the entropy using the number of allowed blocks of length $2$ provides a number greater than $\log 2$. However, by choosing to employ $\|f^3(1)\|$, we obtain
$$
h_{PS}(X_{22}) \leq \dfrac{1}{2^4} \log 941 \approx 0.427934 \approx \log 1.53409. 
$$
By Remark \ref{RemarkCoordinatesChaange} we know that the entropy of $X_{22}$ and $Y = (G,D)$ are the same, and, by Proposition \ref{EntropyMatrixBiggerEntries}, the entropy of $Y$ is at minimum the entropy of $X_{10} = (B,D)$. Using Case 5, we then obtain
$$
h_{PS}(X_{22}) \geq h_{PS}(X_{10}) = \frac{1}{2} h_{PS}(X_5) \geq 0.22376,
$$
which is a more desirable lower bound then the one obtained by Proposition \ref{PropositionAperiodicMatrixImpliesPositiveEntropy}.

Using Remark \ref{RemarkCoordinatesChaangeforEta_i} we can conclude that the entropy of $X_{22}$ and $X_{24}$ indeed coincide.

\section{Existence and non-existence of invariant measures}\label{SectionInvariantMeasures}

This section is dedicated to the study of a specific definition of measure for tree-shifts of the form $X = (A_1, \dots, A_k)$. For simplicity, we consider $k = 2$, since the general case can be easily deduced from it. We begin by introducing some notation. 

Given a transition matrix $M$, let $s(M, i)$ denote the sum of the elements in row $i$ of $M$, and $\overline{M}$ denote the matrix obtained from $M$ by normalizing each row by its sum, that is, for each $i, j$ we have
$$
\overline{M}(i,j) = \dfrac{M(i,j)}{s(M,i)} =: p_M(i,j).
$$
The definition yields, for each $i \in \mathcal{A}$,
\begin{equation}\label{eqNormalizedMatrix}
   \displaystyle\sum\limits_{j=0}^{d-1} p_M(i,j) = 1.
\end{equation}

Let $X = (A_1, A_2)$ be a tree-shift over $\mathcal{A} = \{ 0, \dots, d-1 \}$. An $\sigma_i$-\textit{invariant measure}, $i \in \{a_1, a_2\}$, is a probability $\nu$ in $X$ such that $\nu (C) = \nu ( \sigma_i^{-1} (C))$ for any measurable set $C \subset X$. An \textit{invariant measure} for $X$ is, thus, a probability that is invariant for both $\sigma_{a_1}$ and $\sigma_{a_2}$. 

We say that $v$ is a \textit{probability vector} for $X$ if $v = (v(0), \dots, v(d-1))$, with $v(i) \geq 0$ and $\sum_i v(i) = 1$. Given a vector as such, for each block $C$ in $X$ with nodes $c_x$, $0 \leq |x| \leq n$, define 
\begin{equation}\label{eqDefinitionMeasure}
    \mu([C]) = v (c_\epsilon) \prod_{0 \leq |y| \leq n-1} p_{A_1}(c_y, c_{ya_1}) p_{A_2}(c_y, c_{ya_2}).
\end{equation}
The previous definition is inspired by the context of Markov chains. Below we investigate whether $\mu$ is an invariant measure.

\begin{lemma}\label{LemmaProbabilityVectorDefinesSigma_aInvariantMeasure}
Let $X = (A_1, A_2)$ be a tree-shift and $v$ be a probability vector for $X$. If $v \overline{A}_1 = v$, then the measure $\mu$ defined in equation \eqref{eqDefinitionMeasure} is $\sigma_{a_1}$-invariant.
\end{lemma}

\begin{proof}
Denote by $\sigma_{a_1}^{-1}[C]$ the set of trees $t$ in such a way that $\sigma_{a_1}(t)$ belongs to the cylinder $[C]$. More precisely, if $C$ is a block with length $n$, then
$$
\sigma_{a_1}^{-1}[C] = \left\{ [\Delta] \ ; 
\begin{array}{c}
    \Delta \text{ is a block with length } n+1 \text{ and entries } \Delta_y, \ 0\leq |y| \leq n+1,   \\
     \text{ satisfying } \Delta_{a_1x} = c_x \text{ for all } 0 \leq |x| \leq n 
\end{array}
  \right\}.
$$
Each cylinder in $\sigma_{a_1}^{-1}[C]$ has $2^{n+2}-1$ nodes, $2^{n+1}-1$ of which are fixed by $C$. Thus, $\sigma_{a_1}^{-1}[C]$ is the union of $d^{2^{n+1}}$ cylinders $[\Delta^1], \dots, [\Delta^{d^{2^{n+1}}}]$. In this case, if any $\Delta^j$ is not an admissible block in $X$, then $\mu([\Delta^j]) = 0$, thus, we indeed can consider all such blocks. The equation $v \overline{A}_q = v$ means
\begin{equation}\label{eqProbabilityVectorForM}
   v(i) = \displaystyle\sum\limits_{j=0}^{d-1} v(j) p_{A_1}(j,i).
\end{equation}

We begin by considering that $C$ has length $0$, that is, $[C] = [j]$ for some $j \in \mathcal{A}$. Thus, $\sigma_{a_1}^{-1}[C]$ is the union of cylinders determined by blocks of length $1$ in such a way that node ${a_1}$ is labeled by $j$. Therefore,
$$
\begin{array}{rl}
    \mu(\sigma_{a_1}^{-1}[C]) = & \displaystyle\sum\limits_{r,s=0}^{d-1} v(r) p_{A_1}(r,j) p_{A_2}(r,s) = \displaystyle\sum\limits_{r=0}^{d-1} v(r) p_{A_1}(r,j) \left( \displaystyle\sum_{s=0}^{d-1} p_{A_2}(r,s) \right) \\
    = & \displaystyle\sum\limits_{r=0}^{d-1} v(r) p_{A_1}(r,j) = v(j) = \mu([C]).
\end{array}
$$
Above we used the properties described in equations \eqref{eqNormalizedMatrix} and \eqref{eqProbabilityVectorForM}.

Assume that, for any cylinder $[C]$ of a fixed length $n \geq 0$, it holds $\mu([C]) = \mu(\sigma_{a_1}^{-1}[C])$. This is equivalent to 
\begin{equation}\label{eqInductionHypothesisInvariantMeasure}
    \displaystyle\sum\limits_{j=1}^{d^{2^{n+1}}} \left( v(\Delta_\epsilon^j) p_{A_1}(\Delta_\epsilon^j, c_\epsilon) p_{A_2}(\Delta_\epsilon^j, \Delta_{a_2}^j) \prod_{0 \leq |y| \leq n-1} p_{A_1}(\Delta_{a_2y}^j, \Delta_{a_2ya_1}^j) p_{A_2}(\Delta_{a_2y}^j, \Delta_{a_2ya_2}^j) \right) = v(c_\epsilon),
\end{equation}
where $\Delta^1, \dots, \Delta^{d^{2^{n+1}}}$ are blocks with entries $\Delta_x^j$, $0 \leq |x| \leq n+1$ that determine the cylinders in $\mu(\sigma_{a_1}^{-1}[C])$. Above we proved that equation \eqref{eqInductionHypothesisInvariantMeasure} holds for $n = 0$.

Let $\tilde{C}$ be a block of length $n+1$, with nodes $\tilde{c}_x$ for each $0 \leq |x| \leq n+1$. From previous arguments we have that $\sigma_{a_1}^{-1}[\tilde{C}]$ is the union of $d^{2^{n+2}}$ distinct cylinders, say $[\tilde{\Delta}^1], \dots, [\tilde{\Delta}^{d^{2^{n+2}}}]$. At every level of each of each of these blocks, except from the root, the label at half of the nodes are completely determined by $\tilde{C}$. 

In order to use the induction hypothesis, we group the cylinders of $\sigma_{a_1}^{-1}[\tilde{C}]$ into $d^{2^{n+1}}$ sets: two cylinders in the same set coincide at all nodes, except at level $n+2$. Choose arbitrarily one of these sets and denote its elements as $F^1, \dots, F^{d^{2^{n+1}}}$. Let $f_x^\ell$ be the label of $F^\ell$ at node $x$. By definition, $f_x^1 = \dots = f_x^{d^{2^{n+1}}}$ for all $0 \leq |x| \leq n+1$. Additionally, consider the numbers
$$
r = p_{A_1}(f_\epsilon^1, \tilde{c}_\epsilon) \displaystyle\prod_{0 \leq |x| \leq n} p_{A_1}(\tilde{c}_{x}, \tilde{c}_{xa_1}) \ p_{A_2}(\tilde{c}_{x}, \tilde{c}_{xa_2}),
$$
and
$$
s = p_{A_2}(f_\epsilon^1, f_{a_2}^1)  \displaystyle\prod_{0 \leq |x|\leq n-1} p_{A_1}(f_{a_2x}^1, f_{a_2xa_1}^1) \ p_{A_2}(f_{a_2x}^1, f_{a_2xa_2}^1).
$$

We thus have
$$
\begin{array}{rl}
    \mu \left( \displaystyle\bigcup_{\ell=1}^{d^{2^{n+1}}} [F^\ell] \right) = & \displaystyle\sum\limits_{\ell=1}^{d^{2^{n+1}}} \mu([F^\ell]) \\
    = & \displaystyle\sum\limits_{\ell=1}^{d^{2^{n+1}}} \left(v(f_\epsilon^\ell) \prod_{0\leq |y| \leq n+1} p_{A_1}(f_y^\ell, f_{ya_1}^\ell) \ p_{A_2}(f_y^\ell, f_{ya_2}^\ell)\right) \\
    = & v(f_\epsilon^1)\ r\ s \ \displaystyle\sum_{i,j=0}^{d-1} 
    \left(\displaystyle\prod_{|y|=n} p_{A_1}(f_{a_2y}^1, i) \ p_{A_2}(f_{a_2y}^1, j)\right) \\
    = & v(f_\epsilon^1)\ r\ s \ \displaystyle\prod_{|x|=n} \left( \sum_{i=0}^{d-1} p_{A_1}(f_{a_2x}^1, i) \right) \left( \sum_{j=0}^{d-1} p_{A_2}(f_{a_2x}^1, j)\right) \\
    = & v(f_\epsilon^1) \ r\ s.
\end{array}
$$
The previous arguments prove that the measure of the set $[F^1] \cup \dots \cup [F^{d^{2^{n+1}}}]$ is totally determined by the configuration in common to all blocks $F_1, \dots, F^{d^{2^{n+1}}}$, as expected.

The previously pointed out fact that there exist $d^{2^{n+1}}$ distinct configurations for the first $n+1$ nodes for cylinders in $\sigma_{a_1}^{-1}(\tilde{C})$, say $\{\delta^1_x\}_{0\leq |x|\leq n+1}, \dots, \{\delta^{d^{2^{n+1}}}_x\}_{0\leq |x|\leq n+1}$, the previous computation, and the use of the induction hypothesis in the penultimate equality, lead us to conclude that
$$
\begin{array}{rl}
    \mu(\sigma_{a_1}^{-1}[\tilde{C}]) = & \displaystyle\prod_{0\leq |x| \leq n} p_{A_1}(\tilde{c}_x, \tilde{c}_{xa_1}) \ p_{A_2}(\tilde{c}_x, \tilde{c}_{xa_2}) \\
    & \hspace{-0.1cm} \left[ \displaystyle\sum_{j=1}^{d^{2^{n+2}}} v(\tilde{\Delta}_\epsilon^j)p_{A_1}(\tilde{\Delta}_\epsilon^j, \tilde{c}_\epsilon^j) p_{A_2}(\tilde{\Delta}_\epsilon^j, \tilde{\Delta}_{a_2}^j) \prod_{|y| \leq n} p_{A_1}(\tilde{\Delta}_{a_2y}^j, \tilde{\Delta}_{a_2ya_1}^j) p_{A_2}(\tilde{\Delta}_{a_2y}^j, \tilde{\Delta}_{a_2ya_2}^j) \right] \\
    = & \displaystyle\prod_{0\leq |x| \leq n} p_{A_1}(\tilde{c}_x, \tilde{c}_{xa_1}) \ p_{A_2}(\tilde{c}_x, \tilde{c}_{xa_2}) \\
    &  \hspace{-0.1cm}\left[ \displaystyle\sum_{j=1}^{d^{2^{n+1}}} v(\delta_\epsilon^j)p_{A_1}(\delta_\epsilon^j, \tilde{c}_\epsilon) p_{A_2}(\delta_\epsilon^j, \delta_{a_2}^j) \prod_{0 \leq |y| \leq n-1} p_{A_1}(\delta_{a_2y}^j, \delta_{a_2ya_1}^j) p_{A_2}(\delta_{a_2y}^j, \delta_{a_2ya_2}^j) \right] \\
    = & v(c_\epsilon) \displaystyle\prod_{0\leq |x| \leq n} p_{A_1}(\tilde{c}_x, \tilde{c}_{xa_1}) \ p_{A_2}(\tilde{c}_x, \tilde{c}_{xa_2}) \\
    = & \mu([\tilde{C}]).
\end{array}
$$
It proves the result.
\end{proof}

\begin{proposition}
    Let $X = (A_1, A_2)$ be a tree-shift and $v$ be a probability vector for $X$. Then the measure $\mu$ as defined in equation \eqref{eqDefinitionMeasure} is $\sigma_{a_1}$-invariant if, and only if, $v \overline{A}_1 = v$.
\end{proposition}

\begin{proof}
Lemma \ref{LemmaProbabilityVectorDefinesSigma_aInvariantMeasure} proves one of the implications. Assume that $\mu$ as in equation \eqref{eqDefinitionMeasure} is $\sigma_{a_1}$-invariant. For each $i \in \mathcal{A}$ we have
$$
v(i) = \mu([i]) = \displaystyle\sum_{j,\ell=0}^{d-1} v(j) p_{A_1}(j,i) p_{A_2}(j,\ell) = \displaystyle\sum\limits_{j=0}^{d-1} v(j) p_{A_1}(j,i) \left( \displaystyle\sum\limits_{\ell=0}^{d-1} p_{A_2}(j,k) \right) = \displaystyle\sum\limits_{j=0}^{d-1} v(j) p_{A_1}(j,i).
$$
It proves that $v \overline{A}_1 = v$, as desired.
\end{proof}

The previous arguments can be easily adapted to prove correspondent results considering the matrix $A_2$ instead of $A_1$. Therefore, the following result is immediate.

\begin{proposition}
    Let $X = (A_1, A_2)$ be a tree-shift over $\mathcal{A}$ and $v$ be a probability vector for $\mathcal{A}$. Then $\mu$ as defined in equation \eqref{eqDefinitionMeasure} is invariant if, and only if, $v \overline{A}_1 = v \overline{A}_2 = v$.
\end{proposition}

Below we present a tree-shift that admits invariant measures for both $\sigma_{a_1}$ and $\sigma_{a_2}$, but not for the two dynamics at the same time.

\begin{example}
Consider $X = (A_1, A_2)$, with
$$
A_1 = \begin{pmatrix}
    1 & 1 \\ 1 & 1
\end{pmatrix}
\ \ \ \ \ \ \text{ and } \ \ \ \ \ \ 
A_2 = \begin{pmatrix}
    1 & 1 \\ 1 & 0
\end{pmatrix}.
$$
By definition, we have
$$
\overline{A}_1 = \begin{pmatrix}
    \frac{1}{2} & \frac{1}{2} \\ \frac{1}{2} & \frac{1}{2}
\end{pmatrix}
\ \ \ \ \ \ \text{ and } \ \ \ \ \ \ 
\overline{A}_2 = \begin{pmatrix}
    \frac{1}{2} & \frac{1}{2} \\ 1 & 0
\end{pmatrix},
$$
Since $v_1 = (\frac{1}{2}, \frac{1}{2})$ and $v_2 = (\frac{2}{3}, \frac{1}{3})$ are the unique probability vectors satisfying $v_1 \overline{A}_1 = v_1$ and $v_2 \overline{A}_2 = v_2$, we get that there does not exist an invariant measure $\mu$ for $X$ as in equation \eqref{eqDefinitionMeasure}.
\end{example}

Whenever $A_1 = A_2$, then $X = (A_1, A_2)$ indeed admits an invariant measure.

\section{Topological properties of tree shifts}\label{sectiontopologicalproperties}

In this section, we aim to categorize several tree-shifts according to the definitions of irreducibility, mixing, and chaos in the sense of Devaney exposed in \cite{BanChangChaos}. We present these properties for $k$-trees of a tree-shift $X$ over the alphabet $ \mathcal{A} = \{0, \dots, d-1\} $. Given $x = x_1 \dots x_n \in \Sigma^*$, we denote $A_x = A_{x_1} \dots A_{x_n}$ the product of the matrices, and $\sigma_x = \sigma_{x_n} \circ \dots \circ \sigma_{x_1}$ the composition of shift maps. 

A subset of words $P \subset \Sigma^*$ is called a \textit{prefix set}, so long as any $x, y \in P$ and $w \in \Sigma^*$ do not satisfy $x = yw$. The length of the longest word in $P$, $|P|$, is the length of this set. If every word $x \in \Sigma^*$ with length exceeding $|P|$ has a prefix in $P$, we call it a \textit{complete prefix set} (CPS). We deem that $X$ is \textit{irreducible} if, for each pair of allowed blocks $u$ and $v$ with length $n$, there are $t \in X$ and a complete prefix set $P$ whose words have length at least $n$, in such a way that $u$ is the block of $t$ rooted at $\epsilon$ and, for each $x \in P$, $v$ is the block of $t$ with root at $x$.

The tree-shift $X$ is \textit{mixing} if there exist $ P_{a_1}, \dots, P_{a_k}$ complete prefix sets with the property that, given $u$ and $v$ allowed blocks in $X$ with $|u| = n > 0$, there exists $t \in X$ such that $u$ is the block of $t$ rooted at $\epsilon$ and $v$ is the block of $t$ rooted at $ wx $, for all $ x \in P_{w_n} $, where $ w = w_1 \dots w_n \in \Sigma^n $. By definition, mixing implies irreducibility.

Subsequently we define chaos in the sense of Devaney using the notation of tree-shifts. A \textit{periodic point} of $X$ is a tree $t \in X$ such that $\sigma_x(t) = t $ for all $x$ in some CPS $P$. Moreover, $X$ is \textit{topologically transitive} if, for any open sets $U, V \subset X$, there exists $x \in \Sigma^*$ such that $\sigma_x(U) \cap V \neq \emptyset$. We also say that $X$ has \textit{sensitive dependence on initial conditions} if there exists $\delta > 0$ such that, given $t \in X$, and $V$ which is a neighborhood of $t$, we can find $t^\prime \in V$, and $x \in \Sigma^*$ such that $\dd(\sigma_x(t), \sigma_x(t^\prime)) > \delta$.

Finally, we define $X$ as \textit{chaotic} (in the sense of Devaney) if these three following properties are satisfied:
\begin{enumerate}
    \item[a.] $X$ tree-shift is topologically transitive;
    \item[b.] $X$ has sensitive dependence on initial conditions;
    \item[c.] the set of periodic points of $X$ is dense.
\end{enumerate}
%\textcolor{blue}{
In \cite{BanChangChaos} attention is drawn to the fact that every tree-shift is expanding, so the condition of sensitive dependence on initial conditions would be redundant in the definition of chaos in the sense of Devaney. However, we claim that this statement is inaccurate in the general case. The elaboration of this argument is as follows.%}

%\textcolor{blue}{
A tree-shift $X$ is said to be expanding if, given $t, t^\prime \in X$ in a small neighborhood $U$,
$$
\min\{ \dd(\sigma_a(t), \sigma_a(t^\prime)), \dd(\sigma_b(t), \sigma_b(t^\prime)) \} > \dd(t,t^\prime).
$$
Consider the tree-shift $X_{14}$, defined in terms of the following allowed blocks of length $1$:
\begin{center}
     \begin{forest}
 [$0$[$1$][$1$]]
 \end{forest}
 \ \ \ \ \ \ \ \ 
  \begin{forest}
 [$1$[$0$][$0$]]
 \end{forest}
\end{center}
There are only two elements in this set: the tree $t_0$ with $0$ at its root and at every node $x$ with length $2n$, $n \geq 1$, and $1$ at all nodes with length $2n + 1$, and $t_1$, constructed in a similar way but with $1$ at its root and also at every node with even length. %}
%\textcolor{blue}{
Notice that $\dd(t_0, t_1) = 1$ and this is the greatest possible distance between two elements of any tree-shift, by construction. Moreover, $\dd(\sigma_x(t_0), \sigma_x(t_1)) = \dd(t_1, t_0) = 1$, this can be done using any $\sigma_x$, $x \in \Sigma^*$. Ergo, $X_{14}$ is not expanding, as well as many other examples of tree-shifts consisted only by periodic points. %In some cases, this argument also works for tree-shifts if all of their elements are periodic points.} \textcolor{red}{Isso que eu escrevi está certo? Ou eu estou interpretando e usando errado a definição de expansividade?}

In \cite{BanChangChaos}, Ban and Chang proved a number of conditions to determine whether a tree-shift is irreducible or mixing, among which three will be needed to provide answers to questions evoked in their text. The first is Corollary \ref{Cor3.10BanChang} (that appears in \cite{BanChangChaos} as Corollary 3.11). We added the hypothesis that $X$ has no isolated points, an important condition that the authors did not consider at the time and we discuss in the following paragraph.

\begin{corollary}\label{Cor3.10BanChang}
Suppose $X$ is a tree-shift with no isolated points.
\begin{enumerate}
    \item[a.] If $X$ is an irreducible tree-shift of finite type, then $X$ is chaotic;
    \item[b.] If $X$ is mixing, then $X$ is chaotic.
\end{enumerate}
\end{corollary}

The detail that needs to be taken heed of in the previous result is that the existence of a dense orbit (as assumed for the proof) implies topological transitivity if the dynamical system has no isolated points. In the case that it has, for example, only periodic points, we can not ensure that it is expansive, thus, the sensitive dependence on initial conditions may not occur. For this reason, the irreducibility of a tree-shift of finite type does not implies immediately its chaoticity. For a specific example, consider $X_{14}$. We previously proved that this is a non-expanding tree-shift, and, from Proposition \ref{PropIrreducibleAndMixindTreeShifts}, exposed below, we get that $X_{14}$ is irreducible. The added hypothesis is, therefore, indeed necessary. %for item a. %Since mixing tree-shift of finite type need to have infinitely many points by definition, 

We also present, in Theorem \ref{Theo4.3BanChang} and  Corollary \ref{Cor4.4BanChang}, the two remaining results from \cite{BanChangChaos} needed to our purposes, originally Theorem 4.3 and the first two items of Corollary 4.4, respectively. Although both results are presented by Ban and Chang in terms of a binary tree-shift determined by two transition matrices, we rewrite them considering $k$-trees and the notation of this text. The proofs for $k$-trees are similar to the ones provided in \cite{BanChangChaos}, as pointed out by the authors in their work. % Once more, we need to disregard the case where the tree-shifts have only periodic points, since they can be false in these particular cases.

\begin{theorem}\label{Theo4.3BanChang}
Suppose $X = (A_1, \dots, A_k)$ is a tree-shift. % not consisting only of periodic points.
\begin{enumerate}
    \item[a.] If $X$ is irreducible, then $ A_1, \dots, A_k $ are irreducible;
    \item[b.] $X$ is irreducible if, and only, if for each pair $i,j \in \mathcal{A} $ there exists a CPS $P$ such that $ A_x(i,j) > 0 $ for all $x \in P$;
    \item[c.] $X$ is mixing if, and only if, there exists a CPS $P$ such that $ A_x(i,j) > 0 $ for all $x \in P$ and $i,j \in \mathcal{A}$.
\end{enumerate}
\end{theorem}

\begin{corollary}\label{Cor4.4BanChang}
Consider $ X = (A_1,\dots,A_k) $ a tree-shift consisting of $k$-trees. %, not all of these being periodic points.
\begin{enumerate}
    \item[a.] If $A_1 = \dots = A_k = A$, then $X$ is irreducible if, and only if, $A$ is irreducible;
    \item[b.] If $A_1 = \dots = A_k = A$, then $X$ is mixing if, and only if, $A$ is aperiodic.
\end{enumerate}
\end{corollary}

For further reference, we highlight some properties of the matrices $A-G$ in what follows.

\begin{remark}\label{RemarkMatricesIrreducibleAperiodic}
The matrices $B$, $E$, and $G$ are not irreducible, $C$ is irreducible but not aperiodic and $A$, $D$, and $F$ are aperiodic.
\end{remark}

With these results, we are now able to categorize all the binary tree-shifts over the alphabet $\mathcal{A} = \{0,1\}$ presented in Section \ref{sectionentropytable} in terms of the topological properties previously discussed.

\begin{proposition}\label{PropNotIrreducibleTreeShiftsExamples}
If $X$ is defined by at least one of the matrices $B$, $E$, or $G$, then $X$ is not irreducible.
\end{proposition}

\begin{proof}
The proof is immediate by the counterpositive of Theorem \ref{Theo4.3BanChang}, since the presented matrices are not irreducible.
\end{proof}

Examples 4.5 and 4.12 of \cite{BanChangChaos} prove that $ X_4 $ is mixing and $ X_{21}$ is irreducible. In addition to this, $X_8$ is not irreducible. The following result analyzes the remaining cases.

\begin{proposition}\label{PropIrreducibleAndMixindTreeShifts}
Consider the matrices $A$ - $G$ as defined above.
\begin{enumerate}
    \item[i.] The tree-shifts $X_1, X_6, X_{19}$ and $X_{26}$ are mixing;
    \item[ii.] The tree-shifts $X_3$, $X_{14}$, $X_{15}$, $X_{17}$ and $X_{21}$ are irreducible but not mixing.
\end{enumerate}
\end{proposition}

\begin{proof}
By Remark \ref{RemarkMatricesIrreducibleAperiodic}, $A$, $D$, and $F$ are aperiodic matrices, thus the tree-shifts $X_1, X_{19}$, and $X_{26}$ are mixing according to Corollary \ref{Cor4.4BanChang}. Moreover, considering the CPS $P = \{a, ba, bb\}$, $X_6$ is mixing by Theorem \ref{Theo4.3BanChang}, since
$$ A =
\begin{pmatrix}
1 & 1 \\
1 & 1
\end{pmatrix}
\ \ \ \
FA =
\begin{pmatrix}
1 & 1 \\
2 & 2
\end{pmatrix}
\ \ \ \
FF =
\begin{pmatrix}
1 & 1 \\
1 & 2
\end{pmatrix}.
$$

Let us prove by definition that $X_{14}$ is irreducible. Fix $n \in \N$ and $u$ and $v$ admissible blocks of length $n$. Notice that all nodes at the last level of $u$ have the same label. If this label coincide with the root of $v$, take the CPS $P = \{ x \in \Sigma^* \ ; \ |x| = n \}$ and, if not, consider $P = \{ x \in \Sigma^* \ ; \ |x| = n + 1 \}$. In the first case, we can construct an admissible block with $u$ attached on the root and having $v$ attached to each node at the last level of $u$, and, in the second case, we extend $u$ to a block of length $n+1$ (that coincide with $u$ in the first $n$ levels) and attach $v$ at each node at level $n+1$ of this new block, which is also an admissible block in $X_{14}$. By definition, $X_{14}$ is irreducible.

The previous analysis shows that $X_{14}$ can not be mixing. In fact, the transitions $0 \mapsto 1$ and $1 \mapsto 0$ are possible only after an odd number of steps (or, in other terms, the admissible words in paths at the trees starting with $0$ and ending with $1$ have even length) and the transitions $0 \mapsto 0$ and $1 \mapsto 1$ only occur after an even number of steps (only words $x_1 \dots x_n$ with $n$ odd can have $x_1 = x_n$ to be admissible). Therefore, it is not possible to define two CPS $P_0$ and $P_1$ as in the definition of a mixing tree-shift.

To prove that the remaining tree-shifts presented in item ii. are irreducible we use Theorem \ref{Theo4.3BanChang}. For $X_3$ we can consider the complete prefix sets $P = \{aa, ab, ba, bb\}$ and $\tilde{P} = \{a, b\}$. Taking $i, j = 1, 2$, for all $x \in P$ we get $ A_x (i,i) > 0 $ and, for all $\tilde{x} \in \tilde{P} $ we obtain $A_{\tilde{x}}(i,j) > 0$, $i\neq j$. It implies that $X_3$ is irreducible. The same $P$ and $\tilde{P}$ work for the other cases.

It remains to be proven that tree-shifts $X_3$, $X_{15}$, $X_{17}$, and $X_{21}$ are not mixing. We consider $ X_{21}$ and assert how a similar idea can be applied to the remaining cases. First, notice that, if there exists a CPS $P$ satisfying Theorem \ref{Theo4.3BanChang}, then $a, b \notin P$, since both matrices $D$ and $F$ have a zero entry. We claim that neither element of the form $x = (ab)^j$ or $y = (ab)^j a$ can be in $P$, $j \geq 1$. Indeed, for $x = ab$ and $y = aba$,
$$
A_x = DF =
\begin{pmatrix}
1 & 2 \\
0 & 1
\end{pmatrix}
\ \ \ \ \text{and} \ \ \ \
A_y = DFD =
\begin{pmatrix}
3 & 1 \\
1 & 0
\end{pmatrix}.
$$
Fix $ x = (ab)^j $ with $ j \geq 1 $ arbitrarily. If $ A_x = \begin{pmatrix}  1 & c \\ 0 & 1  \end{pmatrix} $, then
$$ 
A_{(ab)^ja} = A_x D =
\begin{pmatrix}
c+1 & 1 \\
1 & 0
\end{pmatrix}
\ \ \ \ \text{and} \ \ \ \
A_{(ab)^{j+1}} = A_x D F =
\begin{pmatrix}
1 & c+2 \\
0 & 1
\end{pmatrix}.
$$
We then conclude that there is no complete prefix set $P$ for $ X_{21} $ such that $ A_x(i,j) > 0$ for all $i, j = 1, 2$ and all $x \in P$.

For the remaining tree-shifts it is sufficient to notice that $ C^{2n} = B $ and $ C^{2n+1} = C $ for all $n \geq 1$, therefore, there is no $ N \in \N $ such that $ C^N(i,j) > 0 $ for all $i,j$.
\end{proof}

It remains to be determined which tree-shifts are chaotic in the sense of Devaney. By Corollary \ref{Cor3.10BanChang} and Proposition \ref{PropIrreducibleAndMixindTreeShifts} it is immediately observed that $X_1$, $X_3$, $ X_4 $, $ X_6 $, $X_{15}$, $X_{17}$, $ X_{19} $, $X_{21}$, and $X_{26}$ are chaotic. Except for $X_8$, the contrapositive of Corollary \ref{Cor3.10BanChang} guarantees that the tree-shifts that are defined by at least one matrix being $B$, $E$, or $G$ are not mixing. We are also cognizant of $X_{14}$ is not chaotic. Hence, all the tree-shifts considered in this section are covered.

\begin{remark}\label{RemarkTreeShiftsWithoutDensePeriodicPoints}
Tree-shifts defined by at least one matrix with zeros in all entries of a row except from the entry on the diagonal does not have dense periodic points, therefore, are not chaotic.
\end{remark}

Below, we summarize the properties that have been proved in this section, where each $r$ in the first and fifth rows corresponds to the tree-shift $X_r$. The symbol $\checkmark$ means that the property holds, whereas x represents that such property is not verified.

\begin{table}[htp]
\begin{center}
\caption{Topological properties for tree-shifts $X_1$ to $X_{28}$.}
\begin{tabular}{|c|c|c|c|c|c|c|c|c|c|c|c|c|c|c|}
\hline
\multirow{1}{.5em}
& 1 & 2 & 3 & 4 & 5 & 6 & 7 & 8 & 9 & 10 & 11 & 12 & 13 & 14 \\
\hline
irreducible & \checkmark & x & \checkmark & \checkmark & x & \checkmark & x & x & x & x & x & x & x & \checkmark \\
\hline
mixing & \checkmark & x & x & \checkmark & x & \checkmark & x & x & x & x & x & x & x & x  \\
\hline
chaotic & \checkmark & x & \checkmark & \checkmark & x & \checkmark & x & x & x & x & x & x & x & x \\
\hline
\hline
& 15 & 16 & 17 & 18 & 19 & 20 & 21 & 22 & 23 & 24 & 25 & 26 & 27 & 28 \\
\hline
irreducible & \checkmark & x & \checkmark & x & \checkmark & x & \checkmark & x & x & x & x & \checkmark & x & x \\
\hline
mixing & x & x & x & x & \checkmark & x & x & x & x & x & x & \checkmark & x & x \\
\hline
chaotic & \checkmark & x & \checkmark & x & \checkmark & x & \checkmark & x & x & x & x & \checkmark & x & x \\
\hline
\end{tabular}
\end{center}
\end{table}

As previously mentioned, the main importance of these examples is not only to understand the definitions presented in a considerable number of tree-shifts, but also to complement the work of Ban and Chang \cite{BanChangChaos}, from which this section was inspired.

To conclude, we introduce two open problems proposed by Ban and Chang in \cite{BanChangChaos}, relating to entropy:

\textit{Problem 3.} Suppose $ X $ is a tree-shift. Does $ h_{BC}(X) > 0 $ imply the chaos of $ X $?

\textit{Problem 4.} Suppose $ X $ is an irreducible tree-shift of finite type. Is $ h_{BC}(X) > 0 $? Does $ X $ being a mixing tree-shift imply $ h_{BC}(X) > 0 $?

\noindent (The authors also addressed Problems 1 and 2, however, they are outside of the scope of this text.)

Let us first give two examples, with different values of $h_{PS}$, that answer Problem 3 negatively. By Section \ref{sectiontopologicalproperties} and Case 1, $ X_2 $ is an example of non-chaotic tree-shift (in the sense of Devaney) with positive entropy, both $ h_{PS} $ and $ h_{BC} $. Moreover, in Example \ref{ExampleZerohPS-PositivehBC}, we have a tree-shift with zero $h_{PS}$ and positive $h_{BC}$ that is non chaotic, according to Remark \ref{RemarkTreeShiftsWithoutDensePeriodicPoints} (neither mixing or irreducible). 

We consider both questions of Problem 4 separately. The tree-shift $X_{14}$ is irreducible but has zero entropy, so the answer of the first question in Problem 4 is negative. However, we can narrow the scope of the question exclusively to tree-shifts with infinitely many non-isolated trees, and show that, in this case, every irreducible tree-shift of finite type has positive entropy $h_{BC}$ and $h_{PS}$. In this case, since any Markov tree-shift is conjugated to a Markov tree-shift given by transition matrices, let us address the second case. Theorem \ref{Theo4.3BanChang} implies that all matrices that define the Markov tree-shift $X$ need to be irreducible, thus, Proposition \ref{PropositionIrreducibleMatrixNorm2ImpliesPositiveEntropy} guarantees that its entropy is strictly positive. According to Ban and Chang \cite{BanChangChaos}, irreducibility is preserved by conjugation, so we conclude that all irreducible tree-shifts of finite type with infinitely many non-isolated elements have positive $h_{BC}$ and $h_{PS}$. Since a mixing tree-shift of finite type needs to have infinitely many non-isolated points by definition, the answer for the second question of Problem 4 is, therefore, observable. 

Let us discuss another property of the aforementioned dynamical systems. If the tree-shift $X = (A_1, \dots, A_k)$ is mixing, then, by Theorem \ref{Theo4.3BanChang} there exists $p_1 \in \N$ such that $A_1^{p_1}(i,j) > 0$ for all $i, j\in \mathcal{A}$, this implies that $A_1$ is aperiodic. The same argument demonstrates that the remaining matrices have the same property, that is, if $X = (A_1, \dots, A_k)$ is mixing, then $A_1, \dots, A_k$ are aperiodic. Compare with Theorem \ref{Theo4.3BanChang} a. However, the reciprocal is not true: the tree-shift $X_{21}$ is defined in terms of two aperiodic transition matrices but it is not mixing.

Encouraged by the open problems established by Ban and Chang in \cite{BanChangChaos}, we propose the following three problems related to the subject of this work. To the best of our knowledge, the answers to such questions have not yet been obtained. 

\textbf{Problem 1.} Is there a tree-shift conjugacy that is not a $m$-block map, for some $m \in \N$? In other words, how can we establish a conjugation between tree-shifts that is not a block map?

\textbf{Problem 2.} Is the entropy $h_{BC}$ a topological invariant?

\textbf{Problem 3.} Is there a relation between the entropy of a Markov tree-shift whose allowed transitions are given by transition matrices, and the eigenvalues of such matrices, as in the one-dimensional case?


\begin{thebibliography}{000}

\bibitem{AubrunBeal} {\sc N. Aubrun, M.-P. Béal},
{\it Tree-shifts of finite type}, Theoret. Comput. Sci. {\bf 459} (2012), 16-25.


\bibitem{AubrunBealSofic} {\sc N. Aubrun, M.-P. Béal},
{\it Sofic tree-shifts}, Theory Comput. Syst. {\bf 53} (2013), no. 4, 621-644.


\bibitem{BanChangChaos} {\sc J.-C. Ban, C.-H. Chang},
{\it Tree-shifts: Irreducibility, mixing, and the chaos of tree-shifts}, Trans. Amer. Math. Soc. {\bf 369} (2017), no. 12, 8389-8407.


\bibitem{BanChangEntropy} {\sc J.-C. Ban, C.-H. Chang},
{\it Tree-shifts: the entropy of tree-shifts of finite type}, Nonlinearity {\bf 30} (2017), no. 7, p. 2785-2804.


\bibitem{BanChang2020} {\sc J.-C. Ban, C.-H. Chang, W.-G. Hu, G.-Y. Lai, Y.-L. Wu},
{\it Characterization and topological behavior of homomorphism tree-shifts}, Topol. Appl. {\bf 302} (2021), 107848.


\bibitem{BanChangEtAll2021} {\sc J.-C. Ban, C.-H. Chang, W.-G. Hu, Y.-L. Wu},
{\it On structure of topological entropy for tree-shift of finite type}, J. Diff. Equa. {\bf 292} (2021), 325-353


\bibitem{BanChangEtAll2022} {\sc J.-C Ban, C.-H. Chang, W.-G. Hu, Y.-L. Wu},
{\it Topological entropy for shifts of finite type over $\ZZ$ and trees}, Theoret. Comput. Sci. {\bf 903} (2022), 24-32.


\bibitem{BanHuang2023} {\sc J.-C Ban, N.-Z. Huang},
{\it Comutativity of entropy for nonautonomous systems on trees}, Jour. Math. Anal. Appl. {\bf 517903} (2023).


\bibitem{ArtigoProgramacao} {\sc S.-H. Cha}, {\it On the $k$-ary Tree Combinatorics}. Available in \url{https://csis.pace.edu/~scha/CompComb/CSISTR11-284.pdf}.


\bibitem{PetersenSalama1} {\sc K. Petersen, I. Salama,}
{\it Tree shift topological entropy}, Theoret. Comput. Sci. {\bf 473} (2018), 64-71.


\bibitem{PetersenSalama2} {\sc K. Petersen, I. Salama},
{\it Entropy on regular trees}, Discrete Contin. Dyn. Syst. {\bf 40} (2020), no. 7, 4445-4477.
2019.


\bibitem{OEIS} {\it The Online Encyclopedia of Integer Sequences}, https://oeis.org.


\end{thebibliography}
\end{document}